\documentclass[10pt,a4paper,oneside,final]{article}
\usepackage{amsmath}
\usepackage{amssymb}
\usepackage{amsthm}
\usepackage[pdftex]{graphicx}
\usepackage[UKenglish]{babel}

\usepackage[T1]{fontenc}

\newcommand{\R}{{\mathbb{R}}}

\newcommand{\Z}{{\mathbb{Z}}}

\newcommand{\de}{{\mathrm{d}}}
\DeclareMathOperator{\dv}{div}

\newcommand{\diam}{\mathrm{diam}}
\newcommand{\dist}{{\mathrm{dist}}}
\newcommand{\Id}{\mathrm{Id}}
\newcommand{\setchar}[1]{\mathbf{1}_{#1}}

\newcommand{\lebesgue}{\mathcal{L}}

\newcommand{\hd}{\mathcal{H}}
\newcommand{\hdone}{\mathcal{H}^1}

\newcommand{\rca}{\mathrm{rca}}
\newcommand{\fbm}{{\mathrm{fbm}}}
\newcommand{\pushforward}[2]{{{#1}_{\#}#2}}
\newcommand{\restr}{{\mbox{\LARGE$\llcorner$}}}
\newcommand{\spt}{{\mathrm{spt}}}
\newcommand{\weakto}{\rightharpoonup}
\newcommand{\weakstarto}{\stackrel{*}{\rightharpoonup}}
\newcommand{\Wd}[1]{\mathrm{W}_{#1}}
\newcommand{\Wdone}{\Wd{1}}
\newcommand{\AC}{\mathrm{AC}}
\newcommand{\Lip}{\mathrm{Lip}}

\newcommand{\cont}{{\mathcal{C}}}
\newcommand{\contbdd}{\mathcal{C}_{b}}
\newcommand{\smooth}{\mathcal{C}_{c}^{\infty}}
\newcommand{\flux}{{\mathcal{F}}}
\newcommand{\vel}{F}

\newcommand{\reSpace}{\Gamma}
\newcommand{\reMeasure}{P_{\reSpace}}

\newcommand{\TPM}{\mathrm{TPM}}

\newcommand{\brTptEn}{{\mathcal{M}}}
\newcommand{\urbPlEn}{{\mathcal{E}}}
\newcommand{\XiaEn}{M_{\mathrm{F}}}
\newcommand{\MMSEn}{M_{\mathrm{P}}}

\newcommand{\urbPlXia}{E_{\mathrm{F}}}
\newcommand{\urbPlMMS}{E_{\mathrm{P}}}
\newcommand{\urbPlTPM}{C_\Sigma}
\newcommand{\repsilonachi}{r_{\varepsilon,a}^\chi}
\newcommand{\Bcal}{\mathcal{B}}

\newcommand{\Mcal}{\mathcal{M}}
\newcommand{\Ncal}{\mathcal{N}}

\newcommand{\keywords}[1]{\noindent\textbf{Keywords:}\enspace#1}
\newcommand{\subjclass}[1]{\bigskip\noindent\emph{2010 MSC:}\enspace#1}

\numberwithin{equation}{section}

\theoremstyle{plain}
\newtheorem{theorem}{Theorem}[subsection]
\newtheorem{lemma}[theorem]{Lemma}
\newtheorem{proposition}[theorem]{Proposition}

\newtheorem{problem}[theorem]{Problem}

\theoremstyle{definition}
\newtheorem{definition}[theorem]{Definition}

\theoremstyle{remark}
\newtheorem{remark}[theorem]{Remark}

\begin{document}

\title{Equivalent formulations for the branched transport and urban planning problems}
\author{Alessio Brancolini\footnote{Institute for Numerical and Applied Mathematics, University of M\"unster, Einsteinstra\ss{}e 62, D-48149 M\"unster, Germany}\ \footnote{Email address: \texttt{alessio.brancolini@uni-muenster.de}} \and Benedikt Wirth\footnotemark[1]\ \footnote{Email address: \texttt{benedikt.wirth@uni-muenster.de}}}
\date{1st June 2016}
\maketitle

\begin{abstract}
We consider two variational models for transport networks, an \emph{urban planning} and a \emph{branched transport} model,
in both of which there is a preference for networks that collect and transport lots of mass together rather than transporting all mass particles independently.
The strength of this preference determines the ramification patterns and the degree of complexity of optimal networks.
Traditionally, the models are formulated in very different ways,
via cost functionals of the network in case of urban planning or via cost functionals of irrigation patterns or of mass fluxes in case of branched transport.
We show here that actually both models can be described by all three types of formulations;
in particular, the urban planning can be cast into a Eulerian (\emph{flux-based}) or a Lagrangian (\emph{pattern-based}) framework.

\bigskip\keywords{optimal transport, optimal networks, branched transport, irrigation, urban planning, Wasserstein distance}

\subjclass{49Q20, 49Q10, 90B10}

\end{abstract}


\section{Introduction}

The object of this paper are two transport problems, namely the \emph{branched transport problem} and the \emph{urban planning problem}.
In general terms, a transport problem asks how to move mass from a given initial spatial distribution to a specific desired final mass distribution at the lowest possible cost.
Different cost functionals now lead to quite different optimisation problems.

\paragraph{Monge's problem.} The prototype of all these problem is Monge's problem.
In a rather general setting, let $\mu_+, \mu_-$ be finite positive Borel measures on $\R^n$ with the same total mass. The transportation of $\mu_+$ onto $\mu_-$ is modelled by a measurable map $t : \R^n \to \R^n$ such that $\mu_-(B) = \mu_+(t^{-1}(B))$ for all Borel sets $B$. The cost to move an infinitesimal mass $\de\mu_+(x)$ in the point $x$ to the the point $t(x)$ is given by $d(x,t(x))\,\de\mu_+(x)$, and the total cost is then given by the formula
\begin{equation}\label{eq:monge}
 \int_{\R^n} d(x,t(x))\,\de\mu_+(x)\,.
\end{equation}
Usually, $d(x,y) = |x-y|^p$ for some $p\geq1$.

\paragraph{Branched transport problem.}
Monge's cost functional is linear in the transported mass $\de\mu_+(x)$ and thus does not penalise spread out particle movement.
Each particle is allowed to travel independently of the others.
This feature makes the Monge's problem unable to model systems which naturally show ramifications (e.\,g., root systems, leaves, the cardiovascular or bronchial system, etc.).
For this reason, the \emph{branched transport problem} has been introduced by Maddalena, Morel, and Solimini in \cite{Maddalena-Morel-Solimini-Irrigation-Patterns} and by Xia in \cite{Xia-Optimal-Paths}.
It involves a functional which forces the mass to be gathered as much as possible during the transportation. 
This is achieved using a cost which is strictly subadditive in the moved mass so that the marginal transport cost per mass decreases the more mass is transported together (cf.\ Figure\,\ref{fig:transportCost}).
We will formally introduce branched transport later in Section \ref{sec:branchedTransportBrief}.

\paragraph{Urban planning problem.}
The second problem we are interested in is the \emph{urban planning problem}, introduced in \cite{Brancolini-Buttazzo}.
Here, the measures $\mu_+,\mu_-$ have the interpretation of the population and workplace densities, respectively.
In this case the cost depends on the public transportation network $\Sigma$, which is the object of optimisation.
In fact, one part of the cost associated with a network $\Sigma$ is the optimal value of \eqref{eq:monge},
where the cost $d$ depends on $\Sigma$ and is chosen in such a way that transportation on the network is cheaper than outside the network.
The other part models the cost for building and maintaining the network.
A detailed, rigorous description is given in Section \ref{sec:introUrbPlan}.

\paragraph{Patterns and graphs.}
Branched transport has been studied extensively and has several formulations.
Maddalena, Morel, and, Solimini in \cite{Maddalena-Morel-Solimini-Irrigation-Patterns} and Bernot, Caselles, and Morel in \cite{Bernot-Caselles-Morel-Traffic-Plans} proposed a Lagrangian formulation based on \emph{irrigation patterns} $\chi$ that describe the position of each mass particle $p$ at time $t$ by $\chi(p,t)$.
The difference between both articles is that in the latter the particle trajectories cannot be reparameterised without changing the transport cost.
The viewpoint introduced by Xia in \cite{Xia-Optimal-Paths} is instead Eulerian, where only the flux of particles is described, discarding its dependence on the time variable $t$.
A very interesting aspect of branched transport is its regularity theory as studied in several articles, among them \cite{Bernot-Caselles-Morel-Structure-Branched} and \cite{Xia-Interior-Regularity} for the geometric regularity, \cite{Morel-Santambrogio-Regularity} for the regularity of the tangents to the branched structure, \cite{Santambrogio-Landscape} and \cite{Brancolini-Solimini-Hoelder} for the regularity of the landscape function, \cite{Brancolini-Solimini-Fractal} for the fractal regularity of the minimisers.
Equivalence of the different models and formulations are instead the topic of \cite{Maddalena-Solimini-Transport-Distances}, \cite{Maddalena-Solimini-Synchronic}.
Branched transport can also be modelled with curves in the Wasserstein space as in \cite{Brancolini-Buttazzo-Santambrogio}, \cite{Bianchini-Brancolini}, \cite{Brasco-Santambrogio}.

\paragraph{Main result of the paper.}
The main result of this paper is a unified viewpoint of the branched transport problem and the urban planning problem.
Indeed, we show that also the urban planning problem can be cast in the Eulerian or flux-based and in the Lagrangian or pattern-based framework.
This involves the consideration of new functionals which are still subadditive in the moved mass, but not strictly so (cf.\ Figure\,\ref{fig:transportCost}), introducing several technical difficulties.
The main theorem, Theorem\,\ref{thm:urban_plannning_energy_equivalences}, proves the equivalence between the original, the Lagrangian, and the Eulerian formulation of urban planning.
One advantage of these equivalences is that now one can consider regularity questions in the most convenient formulation; as an example we show a single path property of optimal urban planning networks in Proposition\,\ref{prop:single_path_property_for_the_urban_planning_problem}.
For the sake of symmetry we also introduce an additional formulation of branched transport
which is based on the transport network $\Sigma$ as in the original urban planning formulation.
Its equivalence to the pattern- and flux-based formulations is stated in Theorem\,\ref{thm:equivalenceBrTpt}.

The following paragraphs introduce branched transport and urban planning more formally via cost functionals of the transport network $\Sigma$.
Section\,\ref{sec:branched_transport} then recalls the Eulerian and Lagrangian formulation of branched transport and proves their equivalence to the formulation based on $\Sigma$.
Section\,\ref{sec:urban_planning} puts forward the novel Eulerian and Lagrangian formulation of urban planning and states their equivalence,
the proof of which is deferred to the separate Section\,\ref{sec:proof_of_main_theorem}.
Section\,\ref{sec:urban_planning} also states a regularity result for (a subset of all) minimisers, the single path property, based on the model equivalences.

\begin{figure}
\centering
\setlength{\unitlength}{.1\linewidth}
\begin{picture}(3,2.3)
\put(0,0){\includegraphics[width=3\unitlength]{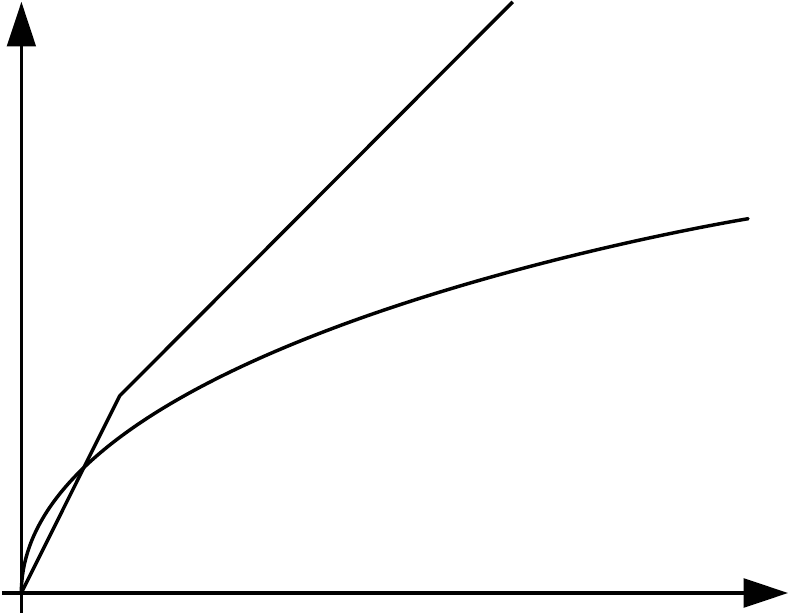}}
\put(2.80,-.15){$m$}
\put(2.9,1.5){$c^\alpha(m)$}
\put(2,2.3){$c^{\varepsilon,a}(m)$}
\end{picture}
\caption{The cost per transport distance is a subadditive function of the transported mass $m$
for the branched transport model ($c^\alpha(m)=m^\alpha$ for some $\alpha\in(0,1)$)
as well as the urban planning model ($c^{\varepsilon,a}(m)=\min(am,m+\varepsilon)$ for some $\varepsilon>0$, $a>1$).}
\label{fig:transportCost}
\end{figure}
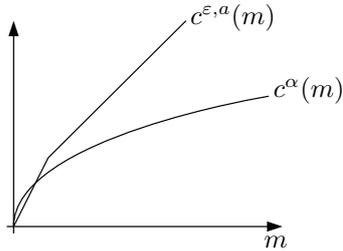

\subsection{Branched transport}\label{sec:branchedTransportBrief}

Branched transport is the first network optimisation problem we consider. As already mentioned, it was introduced in \cite{Maddalena-Morel-Solimini-Irrigation-Patterns,Xia-Optimal-Paths} to model natural and artificial structures which show branching.
Since it is simpler to state, we here take the viewpoint of network optimisation and therefore introduce a new formulation based on the network $\Sigma$.
The original flux- and pattern-based formulations will be introduced in Section\,\ref{thm:equivalenceBrTpt},
and their equivalence to the new formulation will be shown in Theorem\,\ref{thm:equivalenceBrTpt}.

The network is modelled as a one-dimensional set $\Sigma\subset\R^n$, thought of as a collection of pipes through which a quantity with initial distribution $\mu_+$ is transported to achieve the distribution $\mu_-$.
The transport cost scales linearly with transport distance, but is strictly subadditive in the transported mass,
which models scale effects
(i.\,e., transporting an amount of mass in bulk is cheaper
than splitting it up and transporting the portions separately).
Precisely, the cost of moving the mass from $\mu_+$ to $\mu_-$
via a rectifiable pipe network $\Sigma$ is given by
\begin{equation*}
\brTptEn^{\alpha}[\Sigma]
=\inf_{\substack{\vel:\Sigma\to\R^n\setminus\{0\}\\\flux=\vel\hdone\restr\Sigma\\[1ex]\dv\flux=\mu_+-\mu_-}}\int_\Sigma c^\alpha(|\vel|)\,\de\hdone
\qquad\text{with }c^{\alpha}(m)=m^{\alpha}\,.
\end{equation*}

As it will be explained in more detail in Section \ref{sec:branched_transport}, the vector measure $\flux$ describes the mass flux, and $\vel$ denotes its Radon--Nikodym derivative with respect to $\hdone\restr\Sigma$, the one-dimensional Hausdorff measure on $\Sigma$.
The divergence is taken in the distributional sense,
and $\dv\flux=\mu_+-\mu_-$ identifies $\mu_+$ as a flux source and $\mu_-$ as a sink.
The parameter $\alpha\in(0,1)$ governs how strong the scale effect is,
i.\,e., how much cost can be saved by transporting the mass in bulk.
An optimal transport network is a minimizer of the functional $\brTptEn^{\alpha}$.

Existence of minimisers is shown in the Lagrangian or Eulerian framework in \cite{Maddalena-Morel-Solimini-Irrigation-Patterns,Xia-Optimal-Paths,Bernot-Caselles-Morel-Traffic-Plans} and many other works; regularity properties of minimisers are instead considered in \cite{Xia-Interior-Regularity,Bernot-Caselles-Morel-Structure-Branched,Morel-Santambrogio-Regularity,Santambrogio-Landscape,Brancolini-Solimini-Hoelder}; typically minimisers exhibit a type of fractal regularity (see \cite{Brancolini-Solimini-Fractal}).

\subsection{Urban planning}\label{sec:introUrbPlan}

The second energy functional we consider has been introduced as an urban planning model in \cite{Brancolini-Buttazzo}. Here, $\mu_+$ has the interpretation of a population density in an urban region, and $\mu_-$ represents the density of workplaces. The aim is to develop a public transportation network such that the population's daily commute to work is most efficient.

The transportation network is described as a collection of one-dimensional curves; more precisely, it is a set $\Sigma\subset\R^n$ with finite one-dimensional Hausdorff measure $\hdone(\Sigma)$. An employee can travel part of the commute via own means, which is associated with a cost $a>0$ per distance, or via the transportation network $\Sigma$ at the distance-specific cost $b>0$ with $b<a$. Hence, if a travelling path is represented by a curve $\theta : [0,1] \to \R^n$, its cost is given by (for ease of notation we here identify the path $\theta$ with its image $\theta([0,1])$)
\begin{equation*}
a\hdone(\theta\setminus\Sigma)+b\hdone(\theta\cap\Sigma)\,.
\end{equation*}
The minimum cost to get from a point $x\in\R^n$ to a point $y\in\R^n$ is given by the metric
\begin{equation}\label{eq:d_Sigma}
d_\Sigma(x,y)=\inf \{a\hdone(\theta\setminus\Sigma)+b\hdone(\theta\cap\Sigma) \ : \ \theta\in C_{x,y}\}
\end{equation}
where
\begin{equation}\label{eq:C_x_y}
 C_{x,y} = \{\theta : [0,1] \to \R^n \ : \ \theta \ \text{is Lipschitz}, \ \theta(0) = x, \ \theta(1) = y\}
\end{equation}
denotes Lipschitz curves connecting $x$ and $y$. The Wasserstein distance induced by this metric describes the minimum cost to connect the population density $\mu_+$ and the workplace density $\mu_-$ and is given by
\begin{equation}\label{eq:WdSigma}
\Wd{d_\Sigma}(\mu_+,\mu_-)=\inf_{\mu\in\Pi(\mu_+,\mu_-)}\int_{\R^n\times\R^n}d_\Sigma(x,y)\,\de\mu(x,y)\,.
\end{equation}
It is the infimum of formula \eqref{eq:monge} where we choose $d(x,y) = d_\Sigma(x,y)$.
Here, $\Pi(\mu_+,\mu_-)$ denotes the set of transport plans, i.\,e., the set of non-negative finite Borel measures on the product space $\R^n\times\R^n$ whose marginals are $\mu_+$ and $\mu_-$, respectively,
\begin{equation*}
\Pi(\mu_+,\mu_-)=\left\{\mu\in\fbm(\R^n\times\R^n)\,:\,\pushforward{\pi_1}\mu=\mu_+,\pushforward{\pi_2}\mu=\mu_-\right\}\,,
\end{equation*}
where $\pushforward{\pi_i}\mu$ denotes the pushforward of $\mu$ under the projection $\pi_i : \R^n \times \R^n \to \R^n,(x_1,x_2) \mapsto x_i$.

The urban planning problem is the task to find an optimal network $\Sigma$ with respect to the transport cost $\Wd{d_\Sigma}(\mu_+,\mu_-)$ and an additional penalty $\hdone(\Sigma)$, the building and maintaining cost for the network. This leads to the energy functional
\begin{equation*}
\urbPlEn^{\varepsilon,a,b}[\Sigma]=\Wd{d_\Sigma}(\mu_+,\mu_-)+\varepsilon\hdone(\Sigma)\,,
\end{equation*}
to be minimised among all sets $\Sigma\subset\R^n$.
Existence of minimisers has been shown among all closed connected $\Sigma$ (see \cite{Brancolini-Buttazzo} or \cite[Chap.\,3]{BuPrSoSt09}). Without requiring connectedness, existence is proved in \cite[Chap.\,4]{BuPrSoSt09}.

We will actually set $b = 1$ and study $\urbPlEn^{\varepsilon,a}\equiv\urbPlEn^{\varepsilon,a,1}$ without loss of generality, since $\urbPlEn^{\varepsilon,a,b}(\Sigma) = b\urbPlEn^{\frac{\varepsilon}{b},\frac{a}{b},1}(\Sigma)$.

\subsection{Notation and useful notions}

Let us briefly fix some frequently used basic notation.

\begin{itemize}
 \item \textbf{Lebesgue measure.} $\lebesgue^n$ denotes the $n$-dimensional \emph{Lebesgue measure}.

 \item \textbf{Hausdorff measure.} $\hd^r$ denotes the $r$-dimensional \emph{Hausdorff measure}.

 \item \textbf{Non-negative finite Borel measures.} $\fbm(\R^n)$ denotes the set of \emph{non-negative finite Borel measures} on $\R^n$. Notice that these measures are countably additive and also regular by \cite[Thm.\,2.18]{Ru87}. The corresponding total variation norm is denoted by $\|\cdot\|_\fbm$.

 \item \textbf{(Signed or vector-valued) regular countably additive measures.} $\rca(\R^n)$ denotes the set of \emph{(signed or vector-valued) regular countably additive measures} on $\R^n$. The corresponding total variation norm is denoted by $\|\cdot\|_\rca$.

 \item \textbf{Weak-$*$ convergence.} The weak-$*$ convergence on $\fbm(\R^n)$ or $\rca(\R^n)$ is indicated by $\weakstarto$.

 \item \textbf{Restriction of a measure to a set.} Let $(X,\mathcal A,\mu)$ be a measure space and $Y \subset X$ with $Y\in\mathcal A$. The measure $\mu\restr Y$ is the measure defined by
  \begin{displaymath}
  \mu\restr Y(A) = \mu(A \cap Y)\,.
  \end{displaymath}

 \item \textbf{Pushforward of a measure.} For a measure space $(X,\Mcal,\mu)$, a measurable space $(Y,\Ncal)$, and a measurable map $T : X \to Y$, the \emph{pushforward} of $\mu$ under $T$ is the measure $T_\#\mu$ on $(Y,\Ncal)$ defined by
       \begin{displaymath}
        T_\#\mu(B) = \mu(T^{-1}(B)) \quad \text{for all $B \in \Ncal$}.
       \end{displaymath}

 \item \textbf{Continuous and smooth functions.} $\cont_c(\R^n)$ and $\smooth(\R^n)$ denote the set of continuous and smooth functions, respectively, with compact support on $\R^n$.

 \item \textbf{Absolutely continuous functions.} $\AC(I)$ denotes the set of \emph{absolutely continuous functions} on the interval $I\subset\R$.

 \item \textbf{Lipschitz functions.} $\Lip(I)$ denotes the set of \emph{Lipschitz functions} on the compact domain $I$.

 \item \textbf{Characteristic function of a set.} Let $X$ be a set and $A \subseteq X$. The \emph{characteristic function} of the set $A$ is defined as
       \begin{displaymath}
        \setchar{A}:X\to\{0,1\}\,,\quad
        \setchar{A}(x) = \begin{cases}
                          1 & x \in A,\\
                          0 & x \notin A.
                         \end{cases}
       \end{displaymath}

 \item \textbf{Dirac mass.} Let $x \in \R^n$. The \emph{Dirac mass} in $x$ is the distribution $\delta_x$ defined by
 \begin{displaymath}
 \langle \delta_x,\varphi \rangle = \varphi(x) \text{ for all } \varphi \in \smooth(\R^n).
 \end{displaymath}
 The Dirac distribution is the measure $\delta_x(A)=1$ if $x\in A$ and $\delta_x(A)=0$ else.
\end{itemize}

%
%
%
%
%
%
%
%
Finally, for the reader's convenience we compile here a list of the most important symbols with references to the corresponding definitions.

\begin{itemize}
 \item $I = [0,1]$: The unit interval.

 \item $d_\Sigma$, $\Wd{d_\Sigma}$: Urban planning transport metric and transport cost (see \eqref{eq:d_Sigma}-\eqref{eq:WdSigma}).

 \item $C_{x,y}$: Lipschitz paths connecting $x$ and $y$ (see \eqref{eq:C_x_y}).
 
 \item $\flux_G$: Flux associated with a discrete graph $G$ (see \eqref{eqn:graphFlux}).

 \item $(\reSpace,\Bcal(\reSpace),\reMeasure)$: Reference space of all particles (Definition\,\ref{def:reference_space}).

 \item $\chi$: Irrigation pattern of all particles (Definition\,\ref{def:irrigation_pattern}).

 \item $[x]_\chi, m_\chi(x)$: Solidarity class of $x$ and its mass (Definition\,\ref{def:solidarity_classes}).

 \item $\mu_+^\chi, \mu_-^\chi$: Irrigating and irrigated measure (Definition \ref{def:irrigation}).
 
 \item $s_\alpha^\chi$, $\repsilonachi$: Cost densities of branched transport (Definition \ref{eqn:costDensityBrTpt}) and urban planning (Definition \ref{def:urbPlPatternForm}).

 \item $\Theta$: Lipschitz curves on $I$, $\Theta = \Lip(I)$. This notation is introduced in the framework of transport path measures (Definition\,\ref{def:transport_path_measures}).

 \item $\TPM(\mu_+,\mu_-)$: Transport path measures moving $\mu_+$ onto $\mu_-$ (Definition\,\ref{def:transport_path_measures}).
\end{itemize}

%

\section{Branched transport formulations}\label{sec:branched_transport}

In this section we will present the Eulerian or flux-based and the Lagrangian or pattern-based formulations of the \emph{branched transport problem} and state their equivalence to the formulation from Section\,\ref{sec:branchedTransportBrief}. We begin with the Eulerian formulation.

\subsection{Flux-based formulation}

We start considering the formulation given by Xia in \cite{Xia-Optimal-Paths}.

Let $\mu_+ = \sum_{i = 1}^k a_i\delta_{x_i}$, $\mu_- = \sum_{j = 1}^l b_j\delta_{y_j}$ be discrete finite non-negative measures with $a_i, b_j > 0$, $x_i, y_j \in \R^n$.
Suppose also that they have the same mass,
\begin{displaymath}
 \sum_{i = 1}^k a_i = \sum_{j = 1}^l b_j\,.
\end{displaymath}

\begin{remark}
The object of the next definition is called \emph{transport path} in \cite{Xia-Optimal-Paths}, and this is the commonly used term in the branched transport literature.
We deliberately employ the term \emph{mass flux} instead, since it does not only encode a path, but also the amount of mass transported. This way we avoid confusion when referring to actual paths as one-dimensional curves.
\end{remark}

\begin{definition}[Discrete mass flux and cost function]\label{def:mass_flux_cost_function_branched_transport_discrete_case}
A \emph{discrete mass flux between $\mu_+$ and $\mu_-$} is a weighted directed graph $G$ with vertices $V(G)$, straight edges $E(G)$, and edge weight function $w : E(G) \to [0,\infty)$ such that the following conditions are satisfied.
Denoting $e^-$ and $e^+$ the initial and final point of edge $e$, we require the following \emph{mass preserving conditions},
\begin{itemize}
 \item $a_i = \sum_{e \in E(G), e^- = x_i} w(e) - \sum_{e \in E(G), e^+ = x_i} w(e)$ for $i=1,\ldots,k$,
 \item $b_j = \sum_{e \in E(G), e^+ = y_j} w(e) - \sum_{e \in E(G), e^- = y_j} w(e)$ for $j=1,\ldots,l$,
 \item $0 = \sum_{e \in E(G), e^+ = v} w(e) - \sum_{e \in E(G), e^- = v} w(e)$ for $v \in V(G)\setminus\{x_1,\ldots,x_k,y_1,\ldots,y_l\}$.
\end{itemize}
Given a parameter $\alpha\in(0,1)$, we define the transport cost per transport length as $c^\alpha(w) = w^\alpha$ (cf.\ Figure\,\ref{fig:transportCost}).
The \emph{cost function} $\XiaEn^\alpha$ associated with a mass flux $G$ is defined as
\begin{displaymath}
 \XiaEn^\alpha(G) = \sum_{e \in E(G)} c^\alpha(w(e))\,l(e) = \sum_{e \in E(G)} w(e)^\alpha\, l(e)\,,
\end{displaymath}
where $l(e)$ is the length of edge $e$.
\end{definition}

In order to state the branched transport problem in the case of non-discrete finite Borel measures, we need to replace graphs with measures.

\begin{definition}[Graphs as vectorial measures]\label{def:graphs_as_vectorial_measures}
Consider a weighted oriented graph $G$.
Every edge $e\in E(G)$ with direction $\hat e=\frac{e^+-e^-}{|e^+-e^-|}$ can be identified with the vector measure $\mu_e = (\hdone\restr e)\, \hat e$, and the graph can be identified with the vector measure
\begin{equation}\label{eqn:graphFlux}
 \flux_G = \sum_{e \in E(G)} w(e)\mu_e\,.
\end{equation}
All mass preserving conditions satisfied by a discrete mass flux $G$ between $\mu_+$ and $\mu_-$ summarise as $\dv\flux_G = \mu_+ - \mu_-$ (in the distributional sense).
\end{definition}

The identification of graphs with vector measures motivates the definition of a sum operation between graphs that we state here for later usage.
\begin{definition}[Sums of graphs]\label{def:sums_of_graphs}
If $G_1$ and $G_2$ are weighted oriented graphs, then $G_1+G_2$ is unique graph such that
\begin{displaymath}
 \flux_{G_1+G_2} = \flux_{G_1} + \flux_{G_2}.
\end{displaymath}
\end{definition}

\begin{definition}[Continuous mass flux and cost function]
Let $\mu_+,\mu_- \in \fbm(\R^n)$ of equal mass. A vector measure $\flux\in\rca(\R^n;\R^n)$ is a \emph{mass flux} between $\mu_+$ and $\mu_-$, if there exist sequences of discrete measures $\mu_+^k$, $\mu_-^k$ with $\mu_+^k \weakstarto \mu_+$, $\mu_-^k \weakstarto \mu_-$, and a sequence of vector measures $\flux_{G_k}$ with $\flux_{G_k} \weakstarto \flux$, $\dv \flux_{G_k} = \mu_+^k - \mu_-^k$. Note that $\dv\flux=\mu_+-\mu_-$ follows by continuity w.r.t. the weak-$*$ topology.

A sequence $(\mu_+^k,\mu_-^k,\flux_{G_k})$ satisfying the previous properties is called \emph{approximating graph sequence}, and we write $(\mu_+^k,\mu_-^k,\flux_{G_k}) \weakstarto (\mu_+,\mu_-,\flux)$.

If $\flux$ is a mass flux between $\mu_+$ and $\mu_-$, the transport cost $\XiaEn^\alpha$ is defined as
\begin{equation}\label{eq:functional_XiaEn}
 \XiaEn^\alpha(\flux) = \inf\left\{\liminf_{k \to \infty} \XiaEn^\alpha(G_k) \ : \ (\mu_+^k,\mu_-^k,\flux_{G_k}) \weakstarto (\mu_+,\mu_-,\flux)\right\}.
\end{equation}
\end{definition}

\begin{problem}[Branched transport problem, flux formulation]
Given $\mu_+,\mu_- \in \fbm(\R^n)$, the \emph{branched transport problem} is
\begin{equation*}
 \min\{\XiaEn^\alpha(\flux) \ : \ \flux \text{ mass flux between } \mu_+\text{ and }\mu_-\}\,.
\end{equation*}
\end{problem}

\begin{remark}[Existence of minimisers]
A minimiser exists for $\mu_+,\mu_-\in\fbm(\R^n)$ with compact support \cite{Xia-Optimal-Paths}. The minimum value $d_\alpha(\mu_+,\mu_-)$ is a distance on $\fbm(\R^n)$, which induces the weak-$*$ convergence (see \cite{Xia-Optimal-Paths}).
\end{remark}

\begin{remark}\label{rem:flux_equivalent}
It can be shown (see \cite{Xia-Optimal-Paths}, \cite{Bernot-Caselles-Morel-Structure-Branched}) that a mass flux $\flux$ with finite cost can be seen as a rectifiable set $\Sigma$ together with a real multiplicity $\tilde\vel : \Sigma \to (0,\infty)$ and an orientation $\hat e : \Sigma \to \R^n$, $|\hat e| = 1$, such that
\begin{equation*}
 \flux = \tilde\vel \hat e \,(\hdone\restr\Sigma).
\end{equation*}
The quantity $\vel = \tilde\vel \hat e$ describes the mass flux at each point in $\Sigma$.
In that case we have
\begin{displaymath}
 \XiaEn^\alpha(\flux) = \int_\Sigma \tilde\vel^\alpha \de\hdone = \int_\Sigma |\vel|^\alpha \de\hdone\,.
\end{displaymath}
\end{remark}

\subsection{Pattern-based formulation}

In this section we recall the Lagrangian or pattern-based formulation (see \cite{Maddalena-Morel-Solimini-Irrigation-Patterns}, \cite{Bernot-Caselles-Morel-Traffic-Plans}, \cite{Maddalena-Solimini-Synchronic}).

\begin{definition}[Reference space]\label{def:reference_space}
Here we consider a separable uncountable metric space $\reSpace$ endowed with the $\sigma$-algebra $\Bcal(\reSpace)$ of its Borel sets and a positive finite Borel measure $\reMeasure$ with no atoms. We refer to $(\reSpace,\Bcal(\reSpace),\reMeasure)$ as the \emph{reference space}.

The reference space can be interpreted as the space of all particles that will be transported from a distribution $\mu_+$ to a distribution $\mu_-$.
\end{definition}

\begin{remark}
Let $(X,\Mcal,\mu)$ and $(Y,\Ncal,\nu)$ be measure spaces. A map $T : X \to Y$ is said to be an isomorphism of measure spaces, if
\begin{itemize}
 \item $T$ is one-to-one,
 \item for every $N \in \Ncal$, $T^{-1}(N) \in \Mcal$ and $\mu(T^{-1}(N)) = \nu(N)$,
 \item for every $M \in \Mcal$, $T(M) \in \Ncal$ and $\mu(M) = \nu(T(M))$.
\end{itemize}
Recall that if $\reSpace$ is a complete separable metric space and $\reMeasure$ is a positive Borel measure with no atoms (hence $\reSpace$ is uncountable), then $(\reSpace,\Bcal(\reSpace),\reMeasure)$ is isomorphic to the standard space $([0,1],\Bcal([0,1]),m\lebesgue^1 \restr [0,1])$ with $m=\reMeasure(\reSpace)$ (for a proof see \cite[Prop.\,12 or Thm.\,16 in Sec.\,5 of Chap.\,15]{Royden-Real-Analysis} or \cite[Chap.\,1]{Villani-Transport-Old-New}). As a consequence, the following definitions and results are independent of the particular choice of the reference space, and we may assume it to be the standard space without loss of generality.
\end{remark}

\begin{definition}[Irrigation pattern]\label{def:irrigation_pattern}
Let $I = [0,1]$ and $(\reSpace,\Bcal(\reSpace),\reMeasure)$ be our reference space. An \emph{irrigation pattern} is a measurable function $\chi : \reSpace \times I \to \R^n$ such that for almost all $p\in\reSpace$ we have $\chi_p \in \AC(I)$.

A pattern $\tilde\chi$ is \emph{equivalent} to $\chi$ if the images of $\reMeasure$ through the maps $p \mapsto \chi_p, p \mapsto \tilde\chi_p$ are the same. Because of that, a pattern $\chi$ can be regarded as a map $\chi : \reSpace \to \AC(I)$.

For intuition, $\chi_p$ can be viewed as the path followed by the particle $p$. The image of $\chi_p$, that is $\chi_p(I)$, is called a \emph{fibre} and will frequently be identified with the particle $p$.

Here we follow the setting recently introduced in \cite{Maddalena-Solimini-Synchronic}.
\end{definition}

\begin{definition}[Solidarity class]\label{def:solidarity_classes}
For every $x\in\R^n$ we consider the set
\begin{equation}
[x]_\chi = \{q \in \reSpace \ : \ x \in \chi_q(I)\}\label{eq:solidarity_classes}
\end{equation}
of all particles flowing through $x$.
The total \emph{mass} of those particles is given by
\begin{equation*}
m_\chi(x) = \reMeasure([x]_\chi)\,.
\end{equation*}
\end{definition}

\begin{definition}[Cost density, cost functional]\label{eqn:costDensityBrTpt}
For $0 \leq \alpha \leq 1$ we consider the following \emph{cost density},
\begin{equation*}
s_\alpha^\chi(x) = c^\alpha(m_\chi(x))/m_\chi(x) = [m_\chi(x)]^{\alpha - 1},
\end{equation*}
where $c^\alpha$ is the transport cost per transport length from Definition\,\ref{def:mass_flux_cost_function_branched_transport_discrete_case} and we set $s_\alpha^\chi(x) = \infty$ for $m_\chi(x)=0$.
The \emph{cost functional} associated with irrigation pattern $\chi$ is
\begin{equation}\label{eq:functional_MMSEn}
\MMSEn^\alpha(\chi) = \int_{\reSpace \times I} s_\alpha^\chi(\chi_p(t)) |\dot\chi_p(t)|\,\de\reMeasure(p)\, \de t\,.
\end{equation}
The functional $\MMSEn^\alpha$ in the above form has been introduced by Bernot, Caselles, and Morel in \cite{Bernot-Caselles-Morel-Traffic-Plans}.
\end{definition}

\begin{definition}[Irrigating and irrigated measure]\label{def:irrigation}
Let $\chi$ be an irrigation pattern. Let $i_0^\chi,i_1^\chi:\reSpace \to \R^n$ be defined as $i_0^\chi(p) = \chi(p,0)$ and $i_1^\chi(p) = \chi(p,1)$.
The \emph{irrigating measure} and the \emph{irrigated measure} are defined as the pushforward of $\reMeasure$ via $i_0^\chi$ and $i_1^\chi$, respectively,
\begin{displaymath}
 \mu_+^\chi = \pushforward{(i_0^\chi)}{\reMeasure}\,, \quad \mu_-^\chi = \pushforward{(i_1^\chi)}{\reMeasure}\,.
\end{displaymath}
\end{definition}

\begin{problem}[Branched transport problem, pattern formulation]\label{prob:branched_transport_problem}
Given $\mu_+,\mu_- \in \fbm(\R^n)$, the \emph{branched transport problem} is
\begin{equation*}
 \min\{\MMSEn^\alpha(\chi) \ : \ \mu_+^\chi = \mu_+\text{ and }\mu_-^\chi = \mu_-\}\,.
\end{equation*}
\end{problem}

\begin{remark}[Existence of minimisers]
Given $\mu_+,\mu_- \in \fbm(\R^n)$ with compact support, Problem \ref{prob:branched_transport_problem} has a solution \cite{Maddalena-Solimini-Transport-Distances}.
\end{remark}

\subsection{Reparameterisation}

In Definition\,\ref{def:irrigation_pattern} one may equivalently require $\chi_p \in \Lip(I)$ for almost all $p \in \reSpace$, instead of $\chi_p \in \AC(I)$,
which is the content of Propositions\,\ref{prop:fixed_interval_reparameterisation_for_arc-length_parameterised_patterns} and \ref{prop:reparameterised_patterns_have_the_same_cost} below.
This becomes necessary as we will later refer to results from works using either one or the other formulation.
In addition, it allows us to assume Lipschitz continuous fibres throughout the remainder of the article.

Let us first recall the following result, whose proof can be found in \cite[Lem.\,1.1.4]{Ambrosio-Gigli-Sarave-Gradient-Flows}.

\begin{lemma}[Arc-length reparameterisation for $\AC$]\label{lem:arc-length_reparameterisation_for_AC}
Let $v \in \AC([a,b])$ and let $L = \int_a^b |\dot v(t)|\,\de t$ be its length. Let
\begin{align*}
 &\tilde s(t) = \textstyle\int_a^t |\dot v(\tau)|\,\de\tau\,,\\
 &\tilde t(s) = \inf\{t \in [a,b] \ : \ \tilde s(t) = s\}\,,
\end{align*}
then the following holds true,
\begin{itemize}
 \item $\tilde s \in \AC([a,b])$ with $\tilde s(a) = 0$, $\tilde s(b) = L$,
 \item $\tilde v = v \circ\tilde t$ satisfies $v = \tilde v \circ \tilde s$, $\tilde v \in \Lip([0,L])$, and $|\dot{\tilde v}| = 1$ a.\,e.\ in $[0,L]$.
\end{itemize}
\end{lemma}

\begin{proposition}[Arc-length reparameterisation of patterns]\label{prop:arc-length_reparameterisation_for_patterns}
Let $\chi : \reSpace \times I \to \R^n$ be an irrigation pattern. Suppose $\chi$ has finite cost $\MMSEn^\alpha(\chi) < \infty$, and define
\begin{align*}
 \tilde s &: \reSpace \times I \to [0,\infty), && \tilde s(p,t) = \textstyle\int_0^t |\dot\chi(p,\tau)|\,\de\tau\,,\\
 \tilde t &: \reSpace \times [0,\infty) \to I\cup\{\infty\}, && \tilde t(p,s) = \inf\{t \in I \ : \ \tilde s(p,t) = s\}\,,\\
 \tilde\chi&: \reSpace \times [0,\infty) \to \R^n, && \tilde\chi(p,s) = \chi(p,\tilde t(p,s))\,,
\end{align*}
where for notational simplicity we define the infimum of the empty set as $\infty$.
Then, for almost all $p \in \reSpace$ and all $s \in [0,\infty)$, $\tilde\chi(p,\cdot)$ is arc-length parameterised, and $\tilde\chi(\cdot,s)$ is measurable.
\end{proposition}

The proof is similar to the one of \cite[Lem.\,6.2]{Bernot-Caselles-Morel-Traffic-Plans} or \cite[Lem.\,4.1, Lem.\,4.2]{BeCaMo09}. We provide it here for completeness.

\begin{proof}
The fact that $\tilde\chi(p,\cdot)$ is arc-length parameterised for all $p \in \reSpace$ follows from Lemma\,\ref{lem:arc-length_reparameterisation_for_AC}.

Since $\tilde\chi = \chi \circ (\Id_\Gamma \times \tilde t)$, its measurability properties are a consequence of the measurability of $\chi$ and $(\Id_\Gamma \times \tilde t)$ and of the fact that for every null set $N \subset \Gamma \times I$ the set $(\Id_\Gamma \times \tilde t)^{-1}(N)$ is a null set in $\Gamma \times [0,\infty)$.

The measurability of $\Id_\reSpace \times \tilde t$ is proved as in \cite{Bernot-Caselles-Morel-Traffic-Plans} and follows from the measurability of the map $\tilde t$.
We now show that the set $\tilde t^{-1}([0,\lambda])$ is measurable for any $\lambda \in \R$.
Let $\{t_k\}_k$ be a dense sequence in $[0,\infty)$.
Since $\tilde t$ is nondecreasing and lower semicontinuous in the variable $s$, we have
\begin{displaymath}
 \tilde t^{-1}([0,\lambda]) = \bigcap_{h = 1}^\infty \bigcup_{k = 1}^\infty \{p \in \reSpace \ : \ \tilde t(p,t_k) \leq \lambda\} \times \left[0,t_k+\tfrac{1}{h}\right].
\end{displaymath}
Since $\{p \in \reSpace \ : \ \tilde t(p,t_k) \leq \lambda\} = \{p \in \reSpace \ : \ \tilde s(p,\lambda) \geq t_k\}$ is measurable, we obtain that $\tilde t^{-1}([0,\lambda])$ is measurable, too.

Finally, let $N \subset \reSpace \times I$ be a null set, and let $B$ be a Borel set such that $N \subset B$ and $(\reMeasure \otimes \lebesgue^1)(B) \leq \delta$. For almost all $p \in \reSpace$ we have
\begin{displaymath}
 \int_0^\infty \setchar{B}(p,\tilde t(p,s))\,\de s = \int_0^1 \setchar{B}(p,t)\tfrac{\partial\tilde s}{\partial t}(p,t)\,\de t = \int_0^1 \setchar{B}(p,t)|\dot\chi(p,t)|\,\de t\,.
\end{displaymath}
Integrating over $\reSpace$, we obtain
\begin{displaymath}
 (\reMeasure\otimes\lebesgue^1)((\Id_\reSpace \times \tilde t)^{-1}(B)) = \int_\reSpace \int_0^1 \setchar{B}(p,t)|\dot\chi(p,t)|\,\de t\,\de\reMeasure(p)\,.
\end{displaymath}
Due to $\MMSEn^\alpha(\chi) < \infty$, for every $\varepsilon > 0$ there exists $\delta > 0$ such that for every set $B$ with $(\reMeasure\otimes\lebesgue^1)(B) < \delta$ we have
\begin{displaymath}
 \int_B s_\alpha^\chi(\chi(p,t))|\dot\chi(p,t)|\,\de t\,\de\reMeasure(p) < \varepsilon\,.
\end{displaymath}
Since we have that $\setchar{B}(\chi(p,t)) \leq 1 \leq \reMeasure(\reSpace)^{1-\alpha}s_\alpha^\chi(\chi(p,t))$, it follows that
\begin{multline*}
 (\reMeasure\otimes\lebesgue^1)((\Id_\reSpace \times \tilde t)^{-1}(B)) = \int_\reSpace \int_0^1 \setchar{B}(p,t)|\dot\chi(p,t)|\,\de t\,\de\reMeasure(p) \\
 \leq \reMeasure(\reSpace)^{1-\alpha}\int_B s_\alpha^\chi(\chi(p,t))|\dot\chi(p,t)|\,\de t\,\de\reMeasure(p) < \reMeasure(\reSpace)^{1-\alpha}\varepsilon.
\end{multline*}
Choosing $\varepsilon$ arbitrarily small gives $(\reMeasure\otimes\lebesgue^1)((\Id_\reSpace \times \tilde t)^{-1}(N)) = 0$ as desired.
\end{proof}

We may further reparameterise the irrigation pattern.

\begin{proposition}[Constant speed reparameterisation of patterns]\label{prop:fixed_interval_reparameterisation_for_arc-length_parameterised_patterns}
Let $\chi : \reSpace \times I \to \R^n$ be an irrigation pattern with finite cost $\MMSEn^\alpha(\chi)$, let $l(p)=\int_0^1 |\dot\chi(p,t)|\,\de t$ be its fibre length, and let $\tilde\chi$ be as in Proposition\,\ref{prop:arc-length_reparameterisation_for_patterns}.
Then $\hat\chi:\reSpace\times I\to\R^n$, $(p,s) \mapsto \tilde\chi(p,s/l(p))$, is an irrigation pattern which reparameterises the fibres of $\chi$ with $\hat\chi_p\in\Lip(I)$ and constant velocity $|\dot{\hat\chi}_p|$ for almost all $p\in\reSpace$.
\end{proposition}

\begin{proof}
This follows from the properties of $\tilde\chi$.
\end{proof}

\begin{proposition}[Reparameterised patterns have the same cost]\label{prop:reparameterised_patterns_have_the_same_cost}
Let $\chi : \reSpace \times I \to \R^n$ be an irrigation pattern with finite cost $\MMSEn^\alpha(\chi)$ and let $\hat\chi$ be its Lipschitz reparameterisation. Then $\MMSEn^\alpha(\hat\chi)=\MMSEn^\alpha(\chi)$.
\end{proposition}

\begin{proof}
The proof is straightforward, once one notices that the solidarity classes \eqref{eq:solidarity_classes} do not depend on the parameterisation.
\end{proof}

\subsection{Equivalence between the formulations}

It has been proved by Bernot, Caselles, and Morel in \cite[Sec.\,6]{Bernot-Caselles-Morel-Structure-Branched} that the pattern-based formulation is equivalent to the formulation by Xia, even though Xia's formulation does not include the particle motion, while in the pattern-based formulation by Maddalena, Morel, and Solimini the speed of particles occurs in the functional. In particular, minimisers exist for both models, and they can be identified with each other.

\begin{definition}[Branched transport energies]\label{def:branched_transport_energy}
Given two measures $\mu_+,\mu_-\in\fbm(\R^n)$ of equal mass,
for an irrigation pattern $\chi$, a mass flux $\flux$, and a rectifiable set $\Sigma\subset\R^n$ we define
\begin{equation*}
\brTptEn^{\alpha}[\chi]=\MMSEn^{\alpha}(\chi)\,,\quad
\brTptEn^{\alpha}[\flux]=\XiaEn^{\alpha}(\flux)\,,
\end{equation*}
where $\XiaEn^{\alpha}(\flux)$ and $\MMSEn^{\alpha}(\chi)$ are given by \eqref{eq:functional_XiaEn} and \eqref{eq:functional_MMSEn}, respectively, as well as
\begin{align*}
\brTptEn^{\alpha,\mu_+,\mu_-}[\chi]
&=\begin{cases}
\brTptEn^{\alpha}[\chi]&\text{if $\mu_+^\chi = \mu_+$ and $\mu_-^\chi = \mu_-$},\\
\infty&\text{else,}
\end{cases}\\
\brTptEn^{\alpha,\mu_+,\mu_-}[\flux]
&=\begin{cases}
\brTptEn^{\alpha}[\flux]&\text{if }\dv\flux=\mu_+-\mu_-,\\
\infty&\text{else,}
\end{cases}\\
\brTptEn^{\alpha,\mu_+,\mu_-}[\Sigma]
&=\inf \{\brTptEn^{\alpha,\mu_+,\mu_-}[\flux] \ : \ \flux=\vel\hdone\restr\Sigma, \ \vel:\Sigma\to\R^n\setminus\{0\}\}\,.
\end{align*}
The last functional corresponds to the new formulation of Section\,\ref{sec:branchedTransportBrief}.
Note that, if $\Sigma$ is not rectifiable, then $\brTptEn^{\alpha,\mu_+,\mu_-}[\Sigma] = \infty$ (see \cite[Proposition\,4.4]{Xia-Interior-Regularity}).
\end{definition}

\begin{theorem}[Equivalence of branched transport energies]\label{thm:equivalenceBrTpt}
The minimisation problems associated with Definition \ref{def:branched_transport_energy} are equivalent in the sense that
\begin{displaymath}
\min_\chi\brTptEn^{\alpha,\mu_+,\mu_-}[\chi]=\min_\flux\brTptEn^{\alpha,\mu_+,\mu_-}[\flux]=\min_\Sigma\brTptEn^{\alpha,\mu_+,\mu_-}[\Sigma]\,.
\end{displaymath}
The optima can be identified with each other via
\begin{equation*}
\Sigma=\{x\in\R^n\,:\,m_\chi(x)>0\}\,,\quad
\flux=\vel\hdone\restr\Sigma\text{ for the density }\vel=m_\chi\hat e\,,
\end{equation*}
where $\hat e$ is the tangent unit vector to $\Sigma$. Moreover,
\begin{displaymath}
 \int_{\R^n}\varphi\cdot\de\flux=\int_\reSpace\int_I\varphi(\chi_p(t))\cdot\dot\chi_p(t)\,\de t\,\de \reMeasure(p)\text{ for all }\varphi\in\cont_c(\R^n;\R^n)\,.
\end{displaymath}
\end{theorem}

\begin{proof}
The equivalence of the pattern-based formulation $\min_\chi\brTptEn^{\alpha,\mu_+,\mu_-}[\chi]$ to Xia's formulation $\min_\flux\brTptEn^{\alpha,\mu_+,\mu_-}[\flux]$
has been proved by Bernot, Caselles, and Morel (\cite[Sec.\,6]{Bernot-Caselles-Morel-Structure-Branched} or \cite[Chap.\,9]{BeCaMo09}).
Furthermore, for an optimal $\chi$, the set
\begin{displaymath}
 \Sigma = \{x\in\R^n\,:\,m_\chi(x)>0\}\subset\bigcup_{p\in\reSpace}\chi_p(I)
\end{displaymath}
is rectifiable \cite[Lem.\,6.3]{Bernot-Caselles-Morel-Traffic-Plans}, and thus for $\hdone$-a.\,e.\ point $x$ has a tangent unit vector $\hat e(x)$.
Defining a multiplicity via $\tilde\vel=m_\chi$ (see Definition\,\ref{def:solidarity_classes}) we obtain a flux $\flux=\tilde\vel\hat e\hdone\restr\Sigma$ as in Remark\,\ref{rem:flux_equivalent},
and the proof of \cite[Prop.\,9.8]{BeCaMo09} implies that this flux is optimal.

The equality $\min_\flux\brTptEn^{\alpha,\mu_+,\mu_-}[\flux]=\min_\Sigma\brTptEn^{\alpha,\mu_+,\mu_-}[\Sigma]$ follows
by choosing $\Sigma$ as the rectifiable set from Remark\,\ref{rem:flux_equivalent} corresponding to the optimal $\flux$.

Finally, using the relation between the optimal $\flux$ and $\chi$, for $\varphi\in\cont_c(\R^n;\R^n)$ we have
\begin{multline*}
 \int_{\R^n}\varphi\cdot\de\flux
 =\int_{\bigcup_{p\in\reSpace}\chi_p(I)}\varphi(x)\cdot\vel(x)\,\de\hdone(x)\\
 = \int_{\R^n} [x]_\chi\varphi(x)\cdot\hat e(x)\,\de\hdone(x)
 = \int_\reSpace\int_I\varphi(\chi_p(t))\cdot\dot\chi_p(t)\,\de t\,\de \reMeasure(p).
\end{multline*}
This formula follows noting that if two fibres $\chi_p$ and $\chi_q$ coincide in an interval, then their tangents coincide $\hdone$-a.\,e., too.
\end{proof}

\subsection{Regularity properties}

Due to proof of equivalence one can examine regularity properties of minimisers in the most convenient formulation. The following is based on patterns.

\begin{definition}[Loop-free paths and patterns]
Let $\theta\in\Lip(I)$ and let $\chi$ be an irrigation pattern. Following \cite[Def.\,4.5]{Maddalena-Solimini-Synchronic}, we say that \emph{$\theta$ has a loop} if there exist $t_1<t_2<t_3$ such that
\begin{equation*}
\theta(t_1) = \theta(t_3) = x\,,\quad\theta(t_2) \neq x;
\end{equation*}
else we say that $\theta$ is \emph{loop-free}.
$\chi$ is said to be \emph{loop-free} if $\chi_p$ is loop-free for almost all $p\in\reSpace$.
\end{definition}

\begin{definition}[Single path property]\label{def:sigle_path_property}
Let $\chi$ be a loop-free irrigation pattern and let
\begin{displaymath}
 \reSpace_{\overrightarrow{xy}}^\chi = \{p \in \reSpace \ : \ \chi_p^{-1}(x) < \chi_p^{-1}(y)\}\,.
\end{displaymath}
Following \cite[Def.\,3.3]{Bernot-Caselles-Morel-Structure-Branched} and \cite[Def.\,7.3]{BeCaMo09}, $\chi$ has the \emph{single path property} if for every $x,y$ with $\reMeasure(\reSpace_{\overrightarrow{xy}}^\chi) > 0$, the sets $\chi(p,[\chi_p^{-1}(x),\chi_p^{-1}(y)])$ coincide for almost all $p \in \reSpace_{\overrightarrow{xy}}^\chi$.

Note that under the single path property, almost all trajectories from $x$ to $y$ coincide, but they need not coincide as functions of time (since the time variable can be reparameterised).
\end{definition}

\begin{remark}
Optimal patterns are loop-free and enjoy the single path property (see \cite[Sec.\,3]{Bernot-Caselles-Morel-Structure-Branched}, \cite[Chap.\,4]{BeCaMo09} or \cite[Thm.\,4.1]{Maddalena-Solimini-Synchronic}).
\end{remark}

\section{Urban planning formulations}\label{sec:urban_planning}

Here we will employ the same notions as in the previous section to provide the Eulerian or flux-based and the Lagrangian or pattern-based formulations of urban transport. These will then be proved equivalent to the original definition, e.\,g.\ from \cite{BuPrSoSt09}.

\subsection{Flux-based formulation}

Let $G=(V(G),E(G),w)$ be a discrete mass flux between discrete measures $\mu_+,\mu_-\in\fbm(\R^n)$. Let $\Sigma$ be a subgraph of $G$; $\Sigma$ is not required to be connected. Given parameters $\varepsilon > 0$, $a>1$, the \emph{cost function} $\urbPlXia^{\varepsilon,a}$ is defined as
\begin{displaymath}
\urbPlXia^{\varepsilon,a}(G,\Sigma) = \sum_{e \in E(G)\setminus E(\Sigma)} a w(e)l(e) + \sum_{e \in E(\Sigma)} (w(e)+\varepsilon)l(e)\,,
\end{displaymath}
where $l(e)$ is the length of edge $e$. $\urbPlXia^{\varepsilon,a}(G,\Sigma)$ is the cost for employees to travel from an initial distribution $\mu_+$ of homes to a distribution $\mu_-$ of workplaces via the network $G$ using public transport on $\Sigma$. We wish to minimise $\urbPlXia^{\varepsilon,a}(G,\Sigma)$ among admissible pairs $(G,\Sigma)$. For a pair to be optimal one must have
\begin{itemize}
 \item $a w(e) \leq w(e)+\varepsilon$ if $e \in E(G) \setminus E(\Sigma)$, since otherwise the pair $(G,\Sigma\cup\{e\})$ has a lower cost, and
 \item $a w(e) \geq w(e)+\varepsilon$ if $e \in E(\Sigma)$, since else $(G,\Sigma\setminus\{e\})$ has a lower cost.
\end{itemize}
As a result, the cost of an edge $e\in E(G)$ for an optimal $(G,\Sigma)$ is given by $\min(a w(e), w(e)+\varepsilon)$ so that the problem can be the restated with just the mass flux variable $G$.

\begin{definition}[Cost function, flux formulation]
Let $G=(V(G),E(G),w)$ be a discrete mass flux between discrete measures $\mu_+,\mu_-\in\fbm(\R^n)$. Given parameters $\varepsilon > 0$, $a>1$,
we define the transport cost per transport length as $c^{\varepsilon,a}(w) = \min(aw,w+\varepsilon)$ (cf.\ Figure\,\ref{fig:transportCost}).
The \emph{cost function} $\urbPlXia^{\varepsilon,a}$ associated with a mass flux $G$ is defined as
\begin{displaymath}
\urbPlXia^{\varepsilon,a}(G) = \sum_{e \in E(G)}c^{\varepsilon,a}(w(e))\,l(e) = \sum_{e \in E(G)}\min(aw(e),w(e)+\varepsilon) l(e)\,,
\end{displaymath}
where $l(e)$ is the length of edge $e$.

If $\flux\in\rca(\R^n)$ is a general mass flux between general measures $\mu_+,\mu_-\in\fbm(\R^n)$, the \emph{cost function} is defined as
\begin{displaymath}
\urbPlXia^{\varepsilon,a}(\flux) = \inf\{\liminf_{k\to\infty}\urbPlXia^{\varepsilon,a}(G_k) \ : \ (\mu_+^k,\mu_-^k,\flux_{G_k}) \weakstarto (\mu_+,\mu_-,\flux)\}.
\end{displaymath}
\end{definition}

\begin{problem}[Urban planning problem, flux formulation]
Given $\mu_+,\mu_- \in \fbm(\R^n)$, the \emph{urban planning problem} is
\begin{equation*}
 \min\{\urbPlXia^{\varepsilon,a}(\flux) \ : \ \flux \text{ mass flux between } \mu_+\text{ and }\mu_-\}\,.
\end{equation*}
\end{problem}

\begin{remark}[Existence of minimisers]
The existence of mass fluxes with finite cost follows from the existence of irrigation patterns with finite cost (Remark\,\ref{rem:existenceUrbPlFiniteCostPattern}) and Proposition\,\ref{prop:constructFluxFromPattern} later.
Furthermore, $\urbPlXia^{\varepsilon,a}$ is \mbox{weakly-$*$} lower semicontinuous by definition,
and it is bounded below by $\|\cdot\|_\rca$ (since it is the relaxation of a functional, defined only on discrete mass fluxes, which satisfies the same property).
Thus, graphs with uniformly bounded energy are \mbox{weakly-$*$} precompact, and existence of minimisers follows via the direct method of the calculus of variations.
\end{remark}

\begin{remark}
Note that, just like $c^\alpha$ for branched transport, the function $c^{\varepsilon,a}$ is subadditive (cf.\ Figure\,\ref{fig:transportCost}), since a concave function whose graph passes through the origin is subadditive.
This leads to an economy of scales and thus to branched structures.
However, unlike $c^\alpha$ it is not strictly subadditive, so there is a slightly weaker preference for branching structures.
In particular, the minimisers need not be finite graphs away from the support of the initial and final measure,
and mass fluxes can locally be absolutely continuous with respect to Lebesgue measure $\lebesgue^n$ (see Figure\,\ref{fig:urban_planning}).
\end{remark}

\begin{figure}
 \begin{center}
  \setlength{\unitlength}{.3\linewidth}
  \begin{picture}(1.0,0.85)
   \put(0.09,0.05){\includegraphics[width=\unitlength]{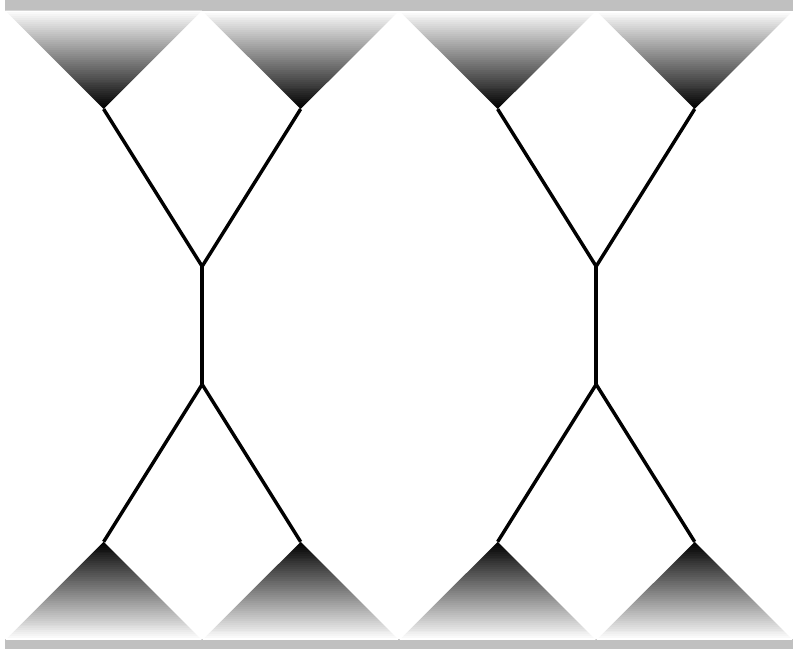}}
   \put(0.5,-0.02){$\mu_+$}
   \put(0.5,0.88){$\mu_-$}
  \end{picture}
  \caption{Sketch of an optimal urban planning mass flux which is absolutely continuous with respect to Lebesgue measure in some regions.
  The grey shade indicates the local flux density.}
  \label{fig:urban_planning}
 \end{center}
\end{figure}

\begin{remark}
For finite graphs, the corresponding optimal network subgraph $\Sigma$ is the graph whose edges are
\begin{displaymath}
E(\Sigma) = \{e \in E(G) \ : \ a w(e) > w(e) + \varepsilon\}\,.
\end{displaymath}
\end{remark}

\subsection{Pattern-based formulation}

\begin{definition}[Cost function, pattern formulation]\label{def:urbPlPatternForm}
Let $(\reSpace,\Bcal(\reSpace),\reMeasure)$ be the reference space and let $\chi : \reSpace \times [0,1] \to \R^n$ be an irrigation pattern. For $\varepsilon > 0$ and $a>1$, consider the density
\begin{equation*}
\repsilonachi(x)= c^{\varepsilon,a}(m_\chi(x))/m_\chi(x) = \begin{cases}
                                                           \min\left(1+\tfrac\varepsilon{m_\chi(x)},a\right) & \text{ if } m_\chi(x) > 0,\\
                                                           a & \text{ if } m_\chi(x) = 0.
                                                          \end{cases}
\end{equation*}
The \emph{cost functional} $\urbPlMMS^{\varepsilon,a}$ is
\begin{equation*}
\urbPlMMS^{\varepsilon,a}(\chi)=\int_{\reSpace\times I}\repsilonachi(\chi_p(t))|\dot\chi_p(t)|\,\de \reMeasure(p)\,\de t\,.
\end{equation*}
\end{definition}

\begin{problem}[Urban planning problem, pattern formulation]
Given $\mu_+,\mu_- \in \fbm(\R^n)$, the \emph{urban planning problem} is
\begin{equation*}
 \min\{\urbPlMMS^{\varepsilon,a}(\chi) \ : \ \mu_+^\chi = \mu_+\text{ and }\mu_-^\chi = \mu_-\}\,.
\end{equation*}
\end{problem}

\begin{remark}[Existence of a finite cost pattern]\label{rem:existenceUrbPlFiniteCostPattern}
An irrigation pattern with finite cost $\urbPlMMS^{\varepsilon,a}$ for a given pair $\mu_+,\mu_-$ of finite Borel measures with same mass and bounded support
can readily be constructed based on the Monge--Kantorovich problem.
Indeed, it is well-known that the 1-Wasserstein distance between $\mu_+$ and $\mu_-$ is bounded (see e.\,g.\ \cite[Chap.\,1]{Villani-Topics-Optimal-Transport}),
\begin{equation*}
\Wd{1}(\mu_+,\mu_-) = \inf_{\mu\in\Pi(\mu_+,\mu_-)} \int_{\R^n\times\R^n} |x - y|\, \de\mu(x,y) < \infty\,,
\end{equation*}
and that the infimum is achieved by a minimising measure $\mu\in\Pi(\mu_+,\mu_-)$, where $\Pi(\mu_+,\mu_-)$ denotes the set of transport plans as in Section\,\ref{sec:introUrbPlan}.

By \cite[Prop.\,12 in Sec.\,5 of Chap.\,15]{Royden-Real-Analysis} there exist a measure $\nu$ on $[0,1]$
and an isomorphism $\varphi:([0,1],\Bcal([0,1]),\nu)\to(\R^n\times\R^n,\Bcal(\R^n\times\R^n),\mu)$ of measure spaces.
Define $\psi:[0,m]\to[0,1]$ as the pseudo-inverse of the cumulative function of $\nu$.
It is clear that the pushforward of the Lebesgue measure under $\varphi\circ\psi$ is $\mu$.
Now take the reference space $(\reSpace,\Bcal(\reSpace),\reMeasure)=([0,m],\Bcal([0,m]),\lebesgue^1\restr [0,m])$ and define the irrigation pattern
\begin{equation*}
\chi(p,t) = C_t(\varphi(\psi(p)))\qquad\text{with }C_t:\R^n\times\R^n\to\R^n\,,\;C_t(x,y)=ty+(1-t)x\,.
\end{equation*}
Since $C_0$ and $C_1$ are the projection on the first and second argument, it is a straightforward exercise to verify that $\mu_+^\chi = \mu_+$, $\mu_-^\chi = \mu_-$. Moreover, we have
\begin{multline*}
 \urbPlMMS^{\varepsilon,a}(\chi) = \int_{\reSpace \times I} \repsilonachi(\chi(p,t))|\dot\chi(p,t)|\,\de\reMeasure(p)\,\de t \leq \int_{\reSpace \times I} a |\dot\chi(p,t)|\,\de\reMeasure(p)\,\de t\\
 = \int_\reSpace a|C_1(\varphi(\psi(p)))-C_0(\varphi(\psi(p)))|\,\de\reMeasure(p) 
 = \int_{\R^n \times \R^n} a|x-y|\,\de\mu(x,y)
 = a\Wdone(\mu_+,\mu_-)\,.
\end{multline*}
\end{remark}

\begin{remark}[Existence of minimisers]
The existence of patterns with minimal urban planning cost will follow from Remark\,\ref{rem:existenceOptPatternUrbPl} via the equivalence of different energy functionals, one of which admits a minimiser.
\end{remark}

Before considering the equivalence between the different formulations, let us state a few properties of the cost functional for later use.

\begin{proposition}[Constant speed reparameterisation of patterns]\label{thm:constSpeedPatternsUrbPl}
Irrigation patterns of finite cost can be reparameterised such that $\chi_p\in\Lip(I)$ and $|\dot\chi_p|$ is constant for almost all $p\in\reSpace$ without changing the cost $\urbPlMMS^{\varepsilon,a}$.
\end{proposition}
\begin{proof}
The proof is analogous to the proofs of Propositions\,\ref{prop:arc-length_reparameterisation_for_patterns} to \ref{prop:reparameterised_patterns_have_the_same_cost},
merely replacing the estimate $1 \leq \reMeasure(\reSpace)^{1-\alpha}s_\alpha^\chi(\chi(p,t))$ by $1 \leq \repsilonachi(\chi(p,t))$.
\end{proof}

The below closely follows \cite[Lemma 4.4, Lemma 4.5]{Maddalena-Solimini-Transport-Distances}.

\begin{definition}[Pointwise convergence]
Let $\chi_n$ be a sequence of irrigation patterns. We say that $\chi_n$ \emph{pointwise converges} to $\chi$ if for almost all $p \in \Gamma$ the curve $\chi_n(p,\cdot)$ converges uniformly to $\chi$.
\end{definition}

\begin{proposition}[$m_\chi$ is u.s.c.\ and $r_{\varepsilon,a}^{\chi}$ is l.s.c.]\label{prop:mass_is_upper_semicontinuous}
Let $\chi_n$ be a sequence of irrigation patterns such that $\chi_n \to \chi$ pointwise. Let $t_n \in I$ such that $t_n \to t$. Then, for almost all $p \in \Gamma$,
\begin{align}
 &m_\chi(\chi(p,t)) \geq \limsup_{n \to \infty} m_{\chi_n}(\chi_n(p,t_n))\,,\label{eq:mass_is_upper_semicontinuous}\\
 &\repsilonachi(\chi(p,t)) \leq \liminf_{n \to \infty} r_{\varepsilon,a}^{\chi_n}(\chi_n(p,t_n))\,.\label{eq:r_is_lower_semicontinuous}
\end{align}
\end{proposition}

\begin{proof}
Fix $p \in \reSpace$ such that $\chi_p \in \AC(I)$ (we only discard a null set of fibres) and define the sets $A$ and $A_n$ as
\begin{displaymath}
 A = \bigcap_n A_n\,, \quad A_n = \bigcup_{k \geq n} [\chi_k(p,t_k)]_{\chi_k}\,.
\end{displaymath}
Recall that $A = \limsup_n A_n$, $\reMeasure([\chi_k(p,t_k)]_{\chi_k}) = m_{\chi_k}(\chi_k(p,t_k))$ and
\begin{displaymath}
 \reMeasure(A) = \lim_n\reMeasure(A_n) \geq \limsup_n m_{\chi_n}(\chi_n(p,t_n))\,.
\end{displaymath}
We want to show that $\reMeasure(A \setminus [\chi(p,t)]_{\chi}) = 0$, that is, $A \subseteq [\chi(p,t)]_{\chi}$ up to a negligible set of fibres so that $m_\chi(\chi(p,t)) \geq \reMeasure(A) \geq \limsup_n m_{\chi_n}(\chi_n(p,t_n))$, proving inequality \eqref{eq:mass_is_upper_semicontinuous}.
Let then $q \in A$ such that $\chi_n(q,\cdot) \to \chi(q,\cdot)$ uniformly (we only discard a null set of fibres of $A$).
Recall that $q \in A$ if and only if there exists an increasing sequence of integers $n_k$ and a sequence $s_{n_k} \in I$ such that $\chi_{n_k}(q,s_{n_k}) = \chi_{n_k}(p,t_{n_k})$.
Suppose now by contraposition that $q \notin [\chi(p,t)]_{\chi}$. Because of the continuity of $\chi$ we have $d=\dist(\chi(p,t),\chi(q,I)) > 0$. Because of the fact that $\chi_n \to \chi$ uniformly, for large $k$ we have also $\dist(\chi(p,t),\chi_{n_k}(q,I)) > \frac d2$, contradicting $\chi_{n_k}(q,s_{n_k}) = \chi_{n_k}(p,t_{n_k})\to\chi(p,t)$.

Inequality \eqref{eq:r_is_lower_semicontinuous} follows immediately from inequality \eqref{eq:mass_is_upper_semicontinuous} and the definition of $\repsilonachi$.
\end{proof}

\begin{proposition}[Lower semicontinuity of $\urbPlMMS^{\varepsilon,a}$]\label{prop:urban_planning_energy_is_lower_semicontinuous}
The functional $\urbPlMMS^{\varepsilon,a}$ is lower semicontinuous with respect to pointwise convergence of patterns.
\end{proposition}

\begin{proof}
Let $\chi_n$ be a sequence of irrigation patterns converging pointwise to the pattern $\chi$. For a given integer $n$ and a given fibre $p$, define $\mu_n = |\dot\chi_n(p,\cdot)|\,\de t$ and $\mu = |\dot\chi(p,\cdot)|\de t$.
As a consequence of the uniform convergence of $\chi_n$ we have $\dot\chi_n(p,\cdot)\to\dot\chi(p,\cdot)$ in the distributional sense and thus
\begin{displaymath}
\mu(A)\leq\liminf_{n\to\infty}\mu_n(A)
\end{displaymath}
for any open $A\subset I$. Thanks to \cite[Def.\,C.1 and Thm.\,C.1]{Maddalena-Solimini-Synchronic} and Proposition\,\ref{prop:mass_is_upper_semicontinuous} we thus have
\begin{displaymath}
\int_I \repsilonachi(\chi(p,t)) \,\de\mu(t) \leq \liminf_{n\to\infty} \int_I r_{\varepsilon,a}^{\chi_n}(\chi_n(p,t)) \,\de\mu_n(t).
\end{displaymath}
Integrating with respect to $\reMeasure$ and applying Fatou's Lemma ends the proof.
\end{proof}

\subsection{Equivalence between the formulations}

\begin{definition}[Urban planning energies]\label{def:urban_plannning_energy}
Given $\mu_+,\mu_-\in\fbm(\R^n)$ of equal mass, for an irrigation pattern $\chi$, a mass flux $\flux$, and a set $\Sigma\subset\R^n$ we define
\begin{equation*}
\urbPlEn^{\varepsilon,a}[\chi]=\urbPlMMS^{\varepsilon,a}(\chi)\,,\quad
\urbPlEn^{\varepsilon,a}[\flux]=\urbPlXia^{\varepsilon,a}(\flux)\,,
\end{equation*}
as well as
\begin{align*}
 \urbPlEn^{\varepsilon,a,\mu_+,\mu_-}[\chi]
 &=\begin{cases}
    \urbPlEn^{\varepsilon,a}[\chi]&\text{if $\mu_+^\chi = \mu_+$ and $\mu_-^\chi = \mu_-$},\\
    \infty&\text{else,}
   \end{cases}\\
 \urbPlEn^{\varepsilon,a,\mu_+,\mu_-}[\flux]
 &=\begin{cases}
    \urbPlEn^{\varepsilon,a}[\flux]&\text{if }\dv\flux=\mu_+-\mu_-,\\
    \infty&\text{else,}
   \end{cases}\\
 \urbPlEn^{\varepsilon,a,\mu_+,\mu_-}[\Sigma]
 &=\begin{cases}
    \Wd{d_\Sigma}(\mu_+,\mu_-)+\varepsilon\hdone(\Sigma)&\text{if }\Sigma\text{ is rectifiable,}\\
    \infty&\text{else,}
   \end{cases}
\end{align*}
with $\Wd{d_\Sigma}$ defined in Section\,\ref{sec:introUrbPlan}.
\end{definition}

\begin{remark}
Recall that $\urbPlEn^{\varepsilon,a,\mu_+,\mu_-}[\Sigma]$ has a minimiser, thanks to \cite[Lem.\,4.10, Lem.\,4.11, Prop.\,4.15, and Thm.\,4.26]{BuPrSoSt09}.
\end{remark}

The next theorem is the main result of this paper. Its proof will be the object of the next section (Section \ref{sec:proof_of_main_theorem}).

\begin{theorem}[Equivalence of urban planning energies]\label{thm:urban_plannning_energy_equivalences}
The minimisation problems in Definition \ref{def:urban_plannning_energy} are equivalent in the sense that, for $\mu_+,\mu_-$ of equal mass and with bounded support, they possess minimisers and satisfy
\begin{equation*}
\min_\chi\urbPlEn^{\varepsilon,a,\mu_+,\mu_-}[\chi]=\min_\flux\urbPlEn^{\varepsilon,a,\mu_+,\mu_-}[\flux]=\min_\Sigma\urbPlEn^{\varepsilon,a,\mu_+,\mu_-}[\Sigma]\,.
\end{equation*}
Similarly to branched transport, there are optima $\chi$, $\flux$, and $\Sigma$ that can be identified with each other via
\begin{gather}
        \int_{\R^n}\varphi\cdot\de\flux=\int_\reSpace\int_I\varphi(\chi_p(t))\cdot\dot\chi_p(t)\,\de t\,\de \reMeasure(p)\text{ for all }\varphi\in\cont_c(\R^n;\R^n)\,,\label{eqn:OptFluxIdentif}\\
        \Sigma=\{x\in\R^n\,:\,m_\chi(x)>\tfrac\varepsilon{a-1}\}\,.\label{eqn:OptSigmaIdentif}
\end{gather}
\end{theorem}

\subsection{Regularity properties}

A consequence of the above-stated equivalence between the different models is the fact that regularity issues can now be considered in the most convenient formulation.
As an example, we state the following single path property of minimisers to the urban planning problem.
Its proof is given at the end of Section\,\ref{sec:EquPatternSet}.

\begin{proposition}[Single path property for the urban planning problem]\label{prop:single_path_property_for_the_urban_planning_problem}
There exists an optimal irrigation pattern $\chi$ for $\urbPlEn^{\varepsilon,a,\mu_+,\mu_-}$ which has the single path property (see Definition \ref{def:sigle_path_property}).
\end{proposition}

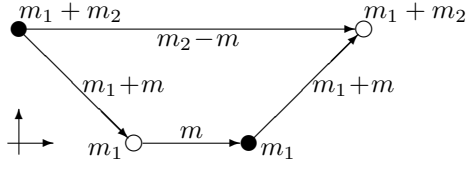
\begin{figure}
\centering
\setlength{\unitlength}{10ex}
\begin{picture}(3,1.1)
\put(-.1,0){\vector(1,0){.4}}
\put(0,-.1){\vector(0,1){.4}}
\put(0,1){\circle*{.15}}
\put(1,0){\circle{.15}}
\put(2,0){\circle*{.15}}
\put(3,1){\circle{.15}}
\put(.6,-.1){$m_1$}
\put(2.1,-.1){$m_1$}
\put(3,1.1){$m_1+m_2$}
\put(0,1.1){$m_1+m_2$}
\put(0,1){\vector(1,-1){.95}}
\put(2,0){\vector(1,1){.95}}
\put(0,1){\vector(1,0){2.925}}
\put(1.075,0){\vector(1,0){.85}}
\put(.55,.45){$m_1\!+\!m$}
\put(2.55,.45){$m_1\!+\!m$}
\put(1.4,.05){$m$}
\put(1.2,.85){$m_2\!-\!m$}
\end{picture}
\caption{Optimal particle motions for the irrigation problem from Remark\,\ref{rem:noSinglePath}
(filled circles denote sources, open ones sinks, each arrow is labelled with its mass flux).
Any $0\leq m\leq m_2$ yields an optimal mass flux or irrigation pattern.}
\label{fig:noSinglePath}
\end{figure}

\begin{remark}\label{rem:noSinglePath}
Unlike for branched transport there exist optimal irrigation patterns not satisfying the single path property, for instance for
\begin{equation*}
\mu_+=(m_1+m_2)\delta_{(0,s)}+m_1\delta_{(2,0)}\,,\quad
\mu_-=m_1\delta_{(1,0)}+(m_1+m_2)\delta_{(3,s)}
\end{equation*}
with $m_1\geq\frac\varepsilon{a-1}\geq m_2$ and $s=\sqrt{a^2-1}$ (see Figure\,\ref{fig:noSinglePath}).
\end{remark}

\section{Proof of Theorem \ref{thm:urban_plannning_energy_equivalences} and of Proposition \ref{prop:single_path_property_for_the_urban_planning_problem}}\label{sec:proof_of_main_theorem}

In this section Theorem \ref{thm:urban_plannning_energy_equivalences} is proved in four steps, corresponding to Propositions\,\ref{prop:constructPatternFromSet}, \ref{prop:constructSetFromPattern}, \ref{prop:constructFluxFromPattern}, and \ref{prop:constructPatternFromFlux}. The proof of \ref{prop:single_path_property_for_the_urban_planning_problem} follows as a corollary of those of Propositions \ref{prop:constructPatternFromSet} and \ref{prop:constructSetFromPattern}. We first introduce some necessary notions concerning measures on the set of paths and cycles in discrete graphs.

\subsection{Transport path measures and other preliminary definitions and results}

\begin{definition}[Transport path measures]\label{def:transport_path_measures}
Let $\Theta=\Lip(I)$ be the set of Lipschitz curves $I \to \R^n$ with the metric
\begin{equation*}
d_\Theta(\theta_1,\theta_2)=\inf\left\{\max_{t\in I}|\theta_1(t)-\theta_2(\varphi(t))|\ :\ \varphi:I\to I\text{ increasing and bijective}\right\}.
\end{equation*}
Following \cite[Def.\,2.5]{BuPrSoSt09}, a \emph{transport path measure} is a measure $\eta$ on $\Theta$ (endowed with the Borel algebra).
If $\mu_+,\mu_- \in \fbm(\R^n)$, we say that the transport path measure $\eta$ moves $\mu_+$ onto $\mu_-$ if
\begin{displaymath}
 \pushforward{p_0}{\eta} = \mu_+,\quad \pushforward{p_1}{\eta} = \mu_-,
\end{displaymath}
where, given $t \in I$, $p_t : \Theta \to \R^n$ is defined by $p_t(\theta) = \theta(t)$. We denote by $\TPM(\mu_+,\mu_-)$ the set of transport path measures moving $\mu_+$ onto $\mu_-$.
\end{definition}

\begin{remark}[Compact sets in $\Theta$]
The following compactness result can be obtained via the Ascoli--Arzel\`a Theorem (see \cite[p.7]{BuPrSoSt09}).
Let $\theta_1,\theta_2,\ldots$ be a sequence in $\Theta$. Suppose that the $\theta_n$ have uniformly bounded lengths and $\theta_n(0) \in \Omega$ for a compact subset $\Omega$ of $\R^n$. Then, the sequence $\theta_n$ is relatively compact in $\Theta$.
\end{remark}

\begin{definition}[Parameterisation of transport path measures]
Given a reference space $(\reSpace,\Bcal(\reSpace),\reMeasure)$,
a \emph{parameterisation} of a transport path measure $\eta$ is a function $\chi : \reSpace \to \Theta$ such that $\pushforward{\chi}{\reMeasure} = \eta$. With little abuse of notation, we write $\chi(p,t)$ instead of $\chi(p)(t)$ (the position of particle $p$ at time $t$).

Notice that the map $\chi : \reSpace \times I \to \R^n$ satisfies
\begin{itemize}
 \item $\chi(p,\cdot) \in \Theta$ for a.\,e.\ $p \in \reSpace$,
 \item $\chi(\cdot,t)$ is measurable for all $t \in I$.
\end{itemize}
\end{definition}

\begin{remark}
Note that a parameterisation of a transport path measure always exists thanks to Skorokhod's Theorem (see \cite[App.\,A, Thm.\,A.3]{BeCaMo09} or \cite[Thm.\,11.7.2]{Dudley}).
\end{remark}

\begin{definition}[Cost of a transport path measure]
Let $a > 1$ and $\Sigma$ be a Borel set with $\hdone(\Sigma) < \infty$. Following \cite[Chap.\,2, eq.\,\mbox{(2.2)}]{BuPrSoSt09}, we define the \emph{cost functional} $\urbPlTPM^{a}$ as
\begin{displaymath}
 \urbPlTPM^{a}(\eta) = \int_\Theta a\hdone(\theta(I)\setminus\Sigma) + \hdone(\theta(I)\cap\Sigma) \,\de \eta(\theta)\,.
\end{displaymath}
\end{definition}

The following is a slight refinement of \cite[Cor.\,2.12 and Prop.\,2.14]{BuPrSoSt09}; we thus only repeat the relevant parts of the proof.

\begin{proposition}[Optimal transport path measures]\label{prop:OptTPM}
Given a bounded Borel set $\Sigma\subset\R^n$ with $\hdone(\Sigma)<\infty$ and $\mu_+,\mu_-\in\fbm(\R^n)$ with equal mass and compact support,
there exists a minimiser $\eta\in \TPM(\mu_+,\mu_-)$ of $\urbPlTPM^{a}$
such that $\eta$-a.\,e.\ path $\theta\in\Theta$ is loop-free.
Furthermore,
\begin{displaymath}
 \urbPlTPM^{a}(\eta) = \Wd{d_\Sigma}(\mu_+,\mu_-)\,.
\end{displaymath}
\end{proposition}
\begin{proof}
Using the notation of \cite[Prop.\,2.14]{BuPrSoSt09}, let us set $\delta_\Sigma(\theta)=a\hdone(\theta(I)\setminus\Sigma) + \hdone(\theta(I)\cap\Sigma)$.
By \cite[Prop.\,2.18]{BuPrSoSt09} it is easy to see that $\delta_\Sigma$ equals its relaxation $\bar\delta_\Sigma$, so that
\begin{equation*}
\urbPlTPM^{a}(\tilde\eta)
=\int_\Theta\delta_\Sigma(\theta)\,\de\tilde\eta(\theta)
=\int_\Theta\bar\delta_\Sigma(\theta)\,\de\tilde\eta(\theta)
=:\overline{C_\Sigma^a}(\tilde\eta)
\end{equation*}
for all transport path measures $\tilde\eta$.
By \cite[Prop.\,2.14]{BuPrSoSt09}, $\overline{C_\Sigma^a}$ and thus also $\urbPlTPM^{a}$ possess a minimiser $\eta\in\TPM(\mu_+,\mu_-)$.

Next, employing the notation from Section\,\ref{sec:introUrbPlan}, for a transport plan $\gamma\in\Pi(\mu_+,\mu_-)$ define $I_\Sigma(\gamma)=\int_{\R^n\times\R^n}d_\Sigma(x,y)\,\de\gamma(x,y)$.
By \cite[Prop.\,2.14]{BuPrSoSt09}, there exist a (minimising) transport plan $\gamma$ and a (minimising) transport path measure $\eta$ such that
\begin{equation*}
\Wd{d_\Sigma}(\mu_+,\mu_-)
=I_\Sigma(\gamma)
=\overline{C_\Sigma^a}(\eta)
=\urbPlTPM^{a}(\eta)\,.
\end{equation*}

Finally, let us show that $\eta$ can be chosen such that $\eta$-a.\,e.\ path is loop-free.
To this end, let $\Omega\subset\R^n$ be a closed ball whose interior contains the support of $\mu_-$ and $\mu_+$ as well as $\Sigma$, and define
\begin{equation*}
\tilde\delta_\Sigma(\theta)=\int_0^1\left(1+(a-1)\setchar{\R^n\setminus\Sigma}(\theta(t))\right)|\dot\theta(t)|\,\de t\,.
\end{equation*}
Note that $\tilde\delta_\Sigma:\Theta\to[0,\infty)$ is lower semicontinuous.
Thus, by the same proof as for \cite[Cor.\,2.11]{BuPrSoSt09},
for any given $x,y\in\Omega$ a minimiser $\theta_{x,y}$ of $\tilde\delta_\Sigma$ exists in $C_{x,y} = \{\theta\in\Theta\ :\ \theta(0)=x,\theta(1)=y\}$ with $\theta_{x,y}(I)\subset\Omega$.
Therefore, by exactly the same proof as for \cite[Cor.\,2.12]{BuPrSoSt09} there is a Borel function $q:\Omega\times\Omega\to\Theta$ such that
\begin{equation*}
\tilde\delta_\Sigma(q(x,y))=\min_{\theta\in C_{x,y}}\tilde\delta_\Sigma(\theta)\,.
\end{equation*}
Now assume there are $x,y\in\Omega$ such that $q(x,y)$ has a loop, that is,
$q(x,y)(t_1)=q(x,y)(t_3)=z$ and $q(x,y)(t_2)\neq z$ for some $0\leq t_1<t_2<t_3\leq1$.
Then
\begin{equation*}
\tilde\delta_\Sigma(q(x,y))>\tilde\delta_\Sigma(\tilde\theta)
\quad\text{for}\quad
\tilde\theta(t)=\begin{cases}q(x,y)(t)&\text{if }t\in[0,t_1]\cup[t_3,1],\\z&\text{else,}\end{cases}
\end{equation*}
which contradicts the optimality of $q(x,y)$.
Thus, $q(x,y)$ is loop-free and therefore injective up to reparameterisation, and we have
\begin{multline*}
\tilde\delta_\Sigma(q(x,y))
=\int_0^1\left(1+(a-1)\setchar{\R^n\setminus\Sigma}(q(x,y)(t))\right)|\dot q(x,y)(t)|\,\de t\\
=\hdone(q(x,y)(I))+(a-1)\hdone(q(x,y)(I)\setminus\Sigma)
=\delta_\Sigma(q(x,y))\,.
\end{multline*}
If we can show $\tilde\delta_\Sigma(q(x,y))=\min_{\theta\in C_{x,y}}\delta_\Sigma(\theta)$,
then, letting $\gamma\in\Pi(\mu_+,\mu_-)$ be an optimal transport plan,
as in the proof of \cite[Prop.\,2.14]{BuPrSoSt09} it follows that $\eta=\pushforward{q}{\gamma}$ is an optimal transport path measure.

We close by proving $\tilde\delta_\Sigma(q(x,y))\leq\min_{\theta\in C_{x,y}}\delta_\Sigma(\theta)$ (the opposite inequality holds trivially).
By \cite[Cor.\,2.11]{BuPrSoSt09}, $\delta_\Sigma=\bar\delta_\Sigma$ possesses a minimiser $\hat\theta\in C_{x,y}$.
If $\hat\theta$ has a loop, that is, $\hat\theta(t_1)=\hat\theta(t_3)=z$ and $\hat\theta(t_2)\neq z$ for some $0\leq t_1<t_2<t_3\leq1$,
and if $\hat\theta(t_2)\notin\hat\theta(I\setminus[t_1,t_3])$, then
\begin{equation*}
\delta_\Sigma(\hat\theta)>\delta_\Sigma(\tilde\theta)
\quad\text{for}\quad
\tilde\theta(t)=\begin{cases}\hat\theta(t)&\text{if }t\in[0,t_1]\cup[t_3,1]\\z&\text{else,}\end{cases}
\end{equation*}
which contradicts the optimality of $\hat\theta$.
Thus, $\hat\theta(I)$ must be homeomorphic to $I$ and can be parameterised by an injective $\theta_{x,y}\in C_{x,y}$.
Therefore we have
\begin{equation*}
\min_{\theta\in C_{x,y}}\delta_\Sigma(\theta)=\delta_\Sigma(\hat\theta)=\delta_\Sigma(\theta_{x,y})=\tilde\delta_\Sigma(\theta_{x,y})\geq\tilde\delta_\Sigma(q(x,y))
\end{equation*}
as desired.
\end{proof}

Finally, let us consider the relation between discrete graphs and transport path measures.

\begin{definition}[Mass flux cycle]
Let $G$ be a discrete mass flux. A \emph{cycle} $C$ is a collection of directed edges $\{e_1,\ldots,e_k\} \subset E(G)$ such that $e_1\cup\ldots\cup e_k$ is homeomorphic to a circle with constant orientation.

The \emph{weight} of the cycle is defined as $w_G(C) = \min_{i=1,\ldots,k}w(e_i)$. Let $J_C^G=\{e\in C\ :\ w_G(e)=w_G(C)\}$ be the set of edges with minimal weight in $G$.

The \emph{$C$-reduced mass flux} $G_C$ is the discrete mass flux such that the set of its edges is $E(G_C)=E(G)\setminus J_C^G$ with weights $w_{G_C}(e)=w_G(e)-w_G(C)$ for $e\in C$ and $w_{G_C}(e)=w_G(e)$ else. Notice that $\dv G_C = \dv G$ so that initial and final measures are unchanged.

The \emph{cycle-reduced mass flux} is the mass flux $G$ reduced by all cycles.
\end{definition}

\begin{remark}
The cycle-reduced mass flux is well-defined since a discrete mass flux has at most finitely many cycles which can be reduced one-by-one.
In doing so, it is easy to see that the reduction order does not matter.
\end{remark}

\begin{lemma}[Cost of mass flux cycle]\label{lem:cycleCost}
Let $G$ be a discrete mass flux between $\mu_+,\mu_-\in\fbm(\R^n)$ and $\tilde G$ the cycle-reduced mass flux. Then
\begin{equation*}
\urbPlXia^{\varepsilon,a}(\tilde G)\leq\urbPlXia^{\varepsilon,a}(G)-\|\flux_{\tilde G}-\flux_G\|_\rca\,.
\end{equation*}
\end{lemma}
\begin{proof}
Let $C$ be a cycle of $G$. We have
\begin{align*}
 \|\flux_{G_C}-\flux_G\|_\rca
  &= \sum_{e\in J_C^G} w_G(e)l(e) + \sum_{e\in E(G_C)} (w_G(e)-w_{G_C}(e))l(e)\\
  &= \sum_{e\in J_C^G} w_G(C)l(e) + \sum_{e\in E(G_C)\cap C} w_G(C)l(e) = w_G(C)\sum_{e\in C}l(e)\,,
\end{align*}
where $l(e)$ denotes the length of edge $e$.
Likewise, since $c^{\varepsilon,a} (w)-c^{\varepsilon,a}(w-w_0) \geq w_0$ for all $w\geq w_0\geq0$,
\begin{align*}
 \urbPlXia^{\varepsilon,a}(G)&-\urbPlXia^{\varepsilon,a}(G_C)\\
 &= \sum_{e\in J_C^G} c^{\varepsilon,a}(w_G(e))l(e) + \sum_{e\in E(G_C)}(c^{\varepsilon,a} (w_G(e))-c^{\varepsilon,a}(w_{G_C}(e)))l(e)\\
 &= \sum_{e\in J_C^G} c^{\varepsilon,a}(w_G(C))l(e) + \sum_{e\in E(G_C)\cap C}(c^{\varepsilon,a} (w_G(e))-c^{\varepsilon,a}(w_G(e)-w_G(C)))l(e)\\
 &\geq w_G(C)\sum_{e\in C}l(e) = \|\flux_{G_C}-\flux_G\|_\rca\,.
\end{align*}
The result now follows repeating this procedure over all cycles and using the additivity of $\|\cdot\|_\rca$ with respect to cycle removal.
\end{proof}

\begin{remark}\label{rem:fluxDecomp}
By \cite[Thm.\,3.5 and Prop.\,3.6]{Ahuja-Magnanti-Orlin} or \cite[Thm.\,1]{Gauthier-Desrosiers-Luebbecke}, a discrete mass flux between $\mu_+,\mu_-\in\fbm(\R^n)$ without cycles can be identified with a transport path measure $\eta$ moving $\mu_+$ to $\mu_-$.
\end{remark}

\newpage
\subsection{Equivalence of the pattern- and set-based formulation}\label{sec:EquPatternSet}

The combination of the following two propositions proves the existence of an optimal pattern $\chi$ and
\begin{displaymath}
\min_\chi\urbPlEn^{\varepsilon,a,\mu_+,\mu_-}[\chi]=\min_\Sigma\urbPlEn^{\varepsilon,a,\mu_+,\mu_-}[\Sigma]
\end{displaymath}
as well as relation\,\eqref{eqn:OptSigmaIdentif} for at least one pair of minimisers, as detailed in Remark\,\ref{rem:existenceOptPatternUrbPl}.

\begin{proposition}\label{prop:constructPatternFromSet}
For any $\Sigma\subset\R^n$ there exists an irrigation pattern $\chi$ such that
\begin{displaymath}
\urbPlEn^{\varepsilon,a,\mu_+,\mu_-}[\chi] \leq \urbPlEn^{\varepsilon,a,\mu_+,\mu_-}[\Sigma].
\end{displaymath}
\end{proposition}
\begin{proof}
By Proposition\,\ref{prop:OptTPM} there exists an optimal transport path measure $\eta \in \TPM(\mu_+,\mu_-)$
such that $\eta$-a.\,e. path $\theta\in\Theta$ is loop-free and
\begin{displaymath}
 \Wd{d_\Sigma}(\mu_+,\mu_-) = \int_\Theta a \hdone(\theta(I)\setminus \Sigma) + \hdone(\theta(I)\cap \Sigma) \,\de \eta(\theta)\,.
\end{displaymath}
Let $\chi$ be a parameterisation of the optimal transport path measure $\eta$ so that $\eta = \pushforward{\chi}{\reMeasure}$.

First we derive the formula
\begin{displaymath}
 \hdone(\tilde\Sigma) = \int_\reSpace \int_{\{t \in I \ : \ \chi(p,t) \in \tilde\Sigma\}} \frac{1}{m_\chi(\chi(p,t))}|\dot\chi(p,t)| \,\de t \,\de \reMeasure(p)\,,
\end{displaymath}
where we introduced $\tilde\Sigma \subseteq \Sigma$ as
\begin{displaymath}
 \tilde\Sigma = \{x \in \Sigma \ : \ m_\chi(x) > 0\}\,.
\end{displaymath}
Since $\chi$ is loop-free (thanks to Proposition \ref{prop:OptTPM}) and thus $\chi_p$ is injective (up to reparameterisation) for a.\,e.\ $p\in\reSpace$, we have
\begin{multline*}
 \int_{\{t\in I\,:\,\chi(p,t) \in \tilde\Sigma\}} \frac{1}{m_\chi(\chi(p,t))}|\dot\chi(p,t)|\,\de t\\
 = \int_{\chi_p(I) \cap \tilde\Sigma} \frac{1}{m_\chi(x)} \,\de\hdone(x)
 = \int_{\R^n} \frac{1}{m_\chi(x)} \setchar{\chi_p(I) \cap \tilde\Sigma}(x) \,\de \hdone(x)\,.
\end{multline*}
Using this as well as the identity $\frac{1}{m_\chi(x)} \int_\reSpace \setchar{\chi_p(I) \cap \tilde\Sigma}(x) \,\de \reMeasure(p) = \setchar{\tilde\Sigma}(x)$\footnote{$\frac{1}{m_\chi(x)} \int_\reSpace \setchar{\chi_p(I) \cap \tilde\Sigma}(x) \,\de \reMeasure(p) = \frac{\reMeasure(\{p \ : \ x \in \chi_p(I)\cap\tilde\Sigma\})}{\reMeasure(\{p \ : \ x \in \chi_p(I)\})} = \setchar{\tilde\Sigma}(x)$},
we obtain
\begin{multline*}
 \int_\reSpace \int_{\{t \in I \ : \ \chi(p,t) \in \tilde\Sigma\}} \frac{1}{m_\chi(\chi(p,t))}|\dot\chi(p,t)| \,\de t \,\de \reMeasure(p)\\
 = \int_\reSpace \int_{\R^n} \frac{1}{m_\chi(x)} \setchar{\chi_p(I) \cap \tilde\Sigma}(x) \,\de \hdone(x) \,\de \reMeasure(p) \\
 = \int_{\R^n} \int_\reSpace \frac{1}{m_\chi(x)} \setchar{\chi_p(I) \cap \tilde\Sigma}(x) \,\de \reMeasure(p) \,\de \hdone(x) = \hdone(\tilde\Sigma)
\end{multline*}
after application of the Fubini--Tonelli Theorem.
Next we notice that
\begin{displaymath}
 \hdone(\chi_p(I)\setminus\Sigma) = \int_{\{\chi_p(t) \notin \Sigma\}} |\dot\chi_p(t)|\,\de t\,, \qquad \hdone(\chi_p(I)\cap\Sigma) = \int_{\{\chi_p(t) \in \Sigma\}} |\dot\chi_p(t)|\,\de t
\end{displaymath}
for a.\,e.\ $p\in\reSpace$ due to the injectivity of $\chi_p$
so that in summary, the urban planning cost can be estimated as
\begin{multline}\label{eqn:urbPlCostEstimate}
 \urbPlEn^{\varepsilon,a,\mu_+,\mu_-}[\Sigma]
 = \int_\Theta a \hdone(\theta(I)\setminus \Sigma) + \hdone(\theta(I) \cap \Sigma) \,\de \eta(\theta) + \varepsilon\hdone(\Sigma) \\
 = \int_\reSpace a \hdone(\chi_p(I)\setminus \Sigma) + \hdone(\chi_p(I) \cap \Sigma) \,\de \reMeasure(p) + \varepsilon\hdone(\Sigma) \\
 \geq \int_\reSpace a \hdone(\chi_p(I)\setminus \Sigma) + \hdone(\chi_p(I) \cap \Sigma) \,\de \reMeasure(p) + \varepsilon\hdone(\tilde\Sigma) \\
 = \int_\reSpace a\int_{\{\chi_p(t) \notin \Sigma\}} |\dot\chi_p(t)|\,\de t + \int_{\{\chi_p(t)\in \Sigma \setminus \tilde\Sigma\}} |\dot\chi_p(t)|\,\de t + \int_{\{\chi_p(t)\in \tilde\Sigma\}} |\dot\chi_p(t)|\,\de t \\
 + \int_{\{\chi_p(t)\in \tilde\Sigma\}} \frac{\varepsilon}{m_\chi(\chi_p(t))}|\dot\chi_p(t)|\,\de t \,\de \reMeasure(p) \\
 = \int_\reSpace a\int_{\{\chi_p(t) \notin \tilde\Sigma\}} |\dot\chi_p(t)|\,\de t + \int_{\{\chi_p(t)\in \tilde\Sigma\}} \left(1+\frac{\varepsilon}{m_\chi(\chi_p(t))}\right)|\dot\chi_p(t)|\,\de t \\
 + \int_{\{\chi_p(t)\in \Sigma\setminus\tilde\Sigma\}} |\dot\chi_p(t)|\,\de t - a\int_{\{\chi_p(t) \in \Sigma\setminus\tilde\Sigma\}} |\dot\chi_p(t)|\,\de t\,\de \reMeasure(p) \\
 \geq \int_{\reSpace \times I} \repsilonachi(\chi_p(t))|\dot\chi_p(t)|\,\de \reMeasure(p) \,\de t + \int_{\Gamma} \int_{\{\chi_p(t)\in \Sigma\setminus\tilde\Sigma\}} (1-a)|\dot\chi_p(t)|\,\de t\,\de \reMeasure(p) \\
 \geq \int_{\reSpace \times I} \repsilonachi(\chi_p(t))|\dot\chi_p(t)|\,\de \reMeasure(p) \,\de t = \urbPlEn^{\varepsilon,a,\mu_+,\mu_-}[\chi]\,.
\end{multline}
Thus the pattern $\chi$ satisfies
\begin{displaymath}
 \urbPlEn^{\varepsilon,a,\mu_+,\mu_-}[\chi] \leq \urbPlEn^{\varepsilon,a,\mu_+,\mu_-}[\Sigma] = \min_\Sigma\urbPlEn^{\varepsilon,a,\mu_+,\mu_-}[\Sigma]\,,
\end{displaymath}
concluding the proof.
\end{proof}

\begin{proposition}\label{prop:constructSetFromPattern}
For any irrigation pattern $\chi$, the rectifiable set $\Sigma = \{x \in \R^n \ : \ m_\chi(x) > \tfrac{\varepsilon}{a-1}\}$ satisfies
\begin{displaymath}
\urbPlEn^{\varepsilon,a,\mu_+,\mu_-}[\chi]\geq\urbPlEn^{\varepsilon,a,\mu_+,\mu_-}[\Sigma]\,.
\end{displaymath}
\end{proposition}
\begin{proof}
Let $\chi$ be an irrigation pattern. Let us define $\eta = \pushforward{\chi}{\reMeasure}$ ($\chi$ here is a map $\reSpace \to \Theta$).
By definition of $\Sigma$, we have
\begin{displaymath}
 \repsilonachi(x) = \begin{cases}
                              1 + \frac{\varepsilon}{m_\chi(x)} & \text{ if } x \in \Sigma,\\
                              a & \text{ otherwise}
                             \end{cases}
\end{displaymath}
and thus
\begin{align*}
 \int_I \repsilonachi(\chi_p(t))|\dot\chi_p(t)|\,\de t
 = \int_{\{\chi_p(t) \in \Sigma\}} \left(1+\tfrac{\varepsilon}{m_\chi(\chi(p,t))}\right)|\dot\chi(p,t)|\,\de t + \int_{\{\chi_p(t) \notin \Sigma\}} a|\dot\chi(p,t)|\,\de t\,.
\end{align*}
First, notice that
\begin{multline*}
 \int_\reSpace \left(\int_{\{\chi_p(t) \in \Sigma\}} |\dot\chi(p,t)|\,\de t + \int_{\{\chi_p(t) \notin \Sigma\}} a|\dot\chi(p,t)|\,\de t\right)\,\de \reMeasure(p) \\
 \geq \int_\Theta a \hdone(\theta(I)\setminus \Sigma) + \hdone(\theta(I) \cap \Sigma) \,\de \eta(\theta) \geq \Wd{d_\Sigma}(\mu_+,\mu_-)\,
\end{multline*}
by Proposition\,\ref{prop:OptTPM}. Furthermore,
\begin{multline*}
 \int_\reSpace \int_{\{\chi_p(t) \in \Sigma\}} \frac{|\dot\chi_p(t)|}{m_\chi(\chi_p(t))}\,\de t\,\de \reMeasure(p)
 \geq \int_\reSpace \int_{\chi_p(I) \cap \Sigma} \frac{1}{m_\chi(x)} \,\de \hdone(x) \,\de \reMeasure(p) \\
 = \int_\Sigma \int_{\{p \in \reSpace \ : \ x \in \chi_p(I)\}} \frac{1}{m_\chi(x)} \,\de \reMeasure(p) \,\de \hdone(x) 
 = \int_\Sigma \frac{m_\chi(x)}{m_\chi(x)} \,\de \hdone(x)
 = \hdone(\Sigma)
\end{multline*}
so that
\begin{multline*}
 \urbPlEn^{\varepsilon,a,\mu_+,\mu_-}[\chi]
 =\int_\reSpace\int_I \repsilonachi(\chi_p(t))|\dot\chi_p(t)|\,\de t \,\de\reMeasure(p)\\
 \geq \Wd{d_\Sigma}(\mu_+,\mu_-)+\varepsilon\hdone(\Sigma)
 = \urbPlEn^{\varepsilon,a,\mu_+,\mu_-}[\Sigma]\,,
\end{multline*}
concluding the proof.
\end{proof}

\begin{remark}\label{rem:existenceOptPatternUrbPl}
Due to the constructive nature of the above proofs, the existence of an optimal $\Sigma$ by \cite[Sec.\,4.4 and Thm.\,4.26]{BuPrSoSt09} implies the existence of an optimal irrigation pattern $\chi$.
Indeed, let $\chi_\Sigma$ be the pattern constructed in Proposition\,\ref{prop:constructPatternFromSet} from an optimal $\Sigma$, then
\begin{multline*}
 \inf_\chi \urbPlEn^{\varepsilon,a,\mu_+,\mu_-}[\chi]
 \leq \urbPlEn^{\varepsilon,a,\mu_+,\mu_-}[\chi_\Sigma]
 \leq \min_{\Sigma} \urbPlEn^{\varepsilon,a,\mu_+,\mu_-}[\Sigma]
 \leq \inf_\chi \urbPlEn^{\varepsilon,a,\mu_+,\mu_-}[\chi]\,,
\end{multline*}
where the last step follows from Proposition\,\ref{prop:constructSetFromPattern}.
Thus, all inequalities are equalities, and $\chi_\Sigma$ is an optimal pattern.
Furthermore, by Proposition\,\ref{prop:constructSetFromPattern}, for any optimal $\chi$ there is an optimal $\Sigma$ satisfying \eqref{eqn:OptSigmaIdentif}.
Note that there are also optimal $\Sigma$ that do not satisfy \eqref{eqn:OptSigmaIdentif} for any optimal irrigation pattern $\chi$,
e.\,g.\ it is easy to see that $\Sigma=\{x \in \R^n \ : \ m_\chi(x) \geq \tfrac{\varepsilon}{a-1}\}$ is such an example.
However, any optimal $\Sigma$ satisfies $$\{x \in \R^n \ : \ m_\chi(x) > \tfrac{\varepsilon}{a-1}\}\subset\Sigma\subset\{x \in \R^n \ : \ m_\chi(x) \geq \tfrac{\varepsilon}{a-1}\}$$ for some optimal $\chi$,
since for an optimal $\Sigma$ the left-hand side and right-hand side in \eqref{eqn:urbPlCostEstimate} must coincide and thus all inequalities must be equalities.
\end{remark}

We end this section proving Proposition \ref{prop:single_path_property_for_the_urban_planning_problem}.

\begin{proof}[Proof of Proposition \ref{prop:single_path_property_for_the_urban_planning_problem}]
The construction of the optimal $\chi$ from an optimal $\Sigma$ in the proof of Proposition\,\ref{prop:constructPatternFromSet} is loop-free.
Furthermore, it may be chosen to have the single path property, since with $\Sigma$ fixed, if there are two paths $\theta_1$ and $\theta_2$ passing through $x$ and $y$,
then without changing the energy $\theta_2$ may be redirected to have the same path as $\theta_1$ between $x$ and $y$.
\end{proof}

\subsection{Equivalence of flux- and pattern-based formulation}

Propositions\,\ref{prop:constructFluxFromPattern} and \ref{prop:constructPatternFromFlux} together prove
\begin{displaymath}
\min_\chi\urbPlEn^{\varepsilon,a,\mu_+,\mu_-}[\chi]=\min_\flux\urbPlEn^{\varepsilon,a,\mu_+,\mu_-}[\flux]
\end{displaymath}
as well as the relation\,\eqref{eqn:OptFluxIdentif} for two minimisers.

\begin{proposition}\label{prop:constructFluxFromPattern}
There exists a mass flux $\flux$ with
\begin{displaymath}
\min_{\chi}\urbPlEn^{\varepsilon,a,\mu_+,\mu_-}[\chi] \geq \urbPlEn^{\varepsilon,a,\mu_+,\mu_-}[\flux]\,.
\end{displaymath}
\end{proposition}
\begin{proof}
For any $h>0$ we consider a discrete grid $Z^h=h\Z^n$ and define discrete approximations $\mu_+^h,\mu_-^h$ of $\mu_+,\mu_-$ via
\begin{displaymath}
\mu_\pm^h=\sum_{i\in\Z^n}\mu_\pm(hi+[0,h]^n)\,\delta_{hi}\,,
\end{displaymath}
where $\delta_{hi}$ is a Dirac mass centred at $hi$.
Due to the bounded support of $\mu_\pm$, $\mu_\pm^h$ is a finite weighted sum of Dirac masses.
Furthermore, $\mu_\pm^h\stackrel*\rightharpoonup\mu_\pm$ as $h\to0$.

Let $\chi$ be an optimal irrigation pattern for $\urbPlEn^{\varepsilon,a,\mu_+,\mu_-}$
and $\chi^h$ an optimal irrigation pattern for $\urbPlEn^{\varepsilon,a,\mu_+^h,\mu_-^h}$.
Further below we will show
\begin{displaymath}
\urbPlEn^{\varepsilon,a,\mu_+,\mu_-}[\chi]\geq\limsup_{h\to0}\urbPlEn^{\varepsilon,a,\mu_+^h,\mu_-^h}[\chi^h]\,.
\end{displaymath}
Furthermore, we will later also show that the $\chi^h$ can be identified with finite graphs (or the corresponding fluxes) $G^h$ such that
\begin{displaymath}
\urbPlEn^{\varepsilon,a,\mu_+^h,\mu_-^h}[\chi^h]=\urbPlXia^{\varepsilon,a}(G^h)\,.
\end{displaymath}
Now denoting by $\Wdone$ the 1-Wasserstein distance, by Remark\,\ref{rem:existenceUrbPlFiniteCostPattern}
we have $\urbPlEn^{\varepsilon,a,\mu_+^h,\mu_-^h}[\chi^h]\leq a\Wdone(\mu_+^h,\mu_-^h)\leq a\mu_+^h(\R^n)(2h+\mathrm{diam}(\spt\mu_+\cup\spt\mu_-))$
so that the finite graphs $G^h$ have uniformly bounded energy.
The corresponding fluxes are also uniformly bounded with respect to the total variation norm due to $\|\flux_{G^h}\|_\rca\leq\urbPlXia^{\varepsilon,a}(G^h)$
and thus are precompact with respect to weak-$*$ convergence.
Hence, there is a mass flux $\flux$ such that $\flux_{G^h}\stackrel*\rightharpoonup\flux$ up to a subsequence.
The lower semicontinuity of the cost then implies
\begin{displaymath}
\urbPlEn^{\varepsilon,a,\mu_+,\mu_-}[\flux]\leq\liminf_{h\to0}\urbPlXia^{\varepsilon,a}(G^h)\,,
\end{displaymath}
which concludes the proof.

In order to show $\urbPlEn^{\varepsilon,a,\mu_+,\mu_-}[\chi]\geq\limsup_{h\to0}\urbPlEn^{\varepsilon,a,\mu_+^h,\mu_-^h}[\chi^h]$,
we associate with every point $x\in\R^n$ the Lipschitz path $I\to\R^n$ given by
\begin{displaymath}
 \theta_x^h(t)=x+t(h\lfloor x/h\rfloor-x)\,,
\end{displaymath}
which connects $x$ with its corresponding point in $Z^h$
(here, $\lfloor c\rfloor=(\lfloor c_1\rfloor,\ldots,\lfloor c_n\rfloor)^T$, where $\lfloor c_i\rfloor=\max\{z\in\Z\ :\ z\leq c_i\}$ is the integer part).
Now we can define new irrigation patterns according to
\begin{equation*}
\chi_+^h(p,\cdot)=\theta_{\chi_p(0)}^h\circ\iota\,,\quad
\chi_-^h(p,\cdot)=\theta_{\chi_p(1)}^h\,,\quad
\tilde\chi^h(p,t)=\begin{cases}\chi_+^h(p,3t)&\text{if }t\in[0,\frac13],\\\chi(p,3t-1)&\text{if }t\in(\frac13,\frac23],\\\chi_-^h(p,3t-2)&\text{if }t\in(\frac23,1],\end{cases}
\end{equation*}
where $\iota(t) = 1-t$.
It can easily be checked that $\tilde\chi^h$ transports $\mu_+^h$ to $\mu_-^h$ with cost
\begin{align*}
\urbPlEn^{\varepsilon,a,\mu_+^h,\mu_-^h}[\tilde\chi^h]
&\leq\urbPlMMS^{\varepsilon,a}(\chi_+^h)+\urbPlMMS^{\varepsilon,a}(\chi)+\urbPlMMS^{\varepsilon,a}(\chi_-^h)\\
&=\urbPlMMS^{\varepsilon,a}(\chi_+^h)+\urbPlEn^{\varepsilon,a,\mu_+,\mu_-}[\chi]+\urbPlMMS^{\varepsilon,a}(\chi_-^h)\,.
\end{align*}
The estimate $\urbPlMMS^{\varepsilon,a}(\chi_\pm^h)\leq\mu_\pm(\R^n)ah\sqrt n$ (since all paths in $\chi_\pm^h$ have length no longer than $h\sqrt n$)
as well as $\urbPlEn^{\varepsilon,a,\mu_+^h,\mu_-^h}[\chi^h]\leq\urbPlEn^{\varepsilon,a,\mu_+^h,\mu_-^h}[\tilde\chi^h]$ then yield the desired result.

Finally, we need to show that the $\chi^h$ can be identified with finite graphs.
For $i,j\in\Z^n$ let $\reSpace_{ij} = \{p \in \reSpace \ : \ \chi_p^h(0)=hi, \chi_p^h(1)=hj\}$.
We have (potentially changing $\chi^h$ on a $\reMeasure$-null set, which does not alter its cost)
\begin{displaymath}
\reSpace=\bigcup_{i,j\in\Z^n}\reSpace_{ij}\,,
\end{displaymath}
where only finitely many, say $N$, terms of this union are nonempty,
since $\mu_+^h$ and $\mu_-^h$ consist of only finitely many weighted Dirac measures.
Since $\chi^h$ may be assumed to have the single path property (see Proposition\,\ref{prop:single_path_property_for_the_urban_planning_problem}),
$\chi^h:\reSpace\to \Lip(I)$ may be taken constant on each nonempty $\reSpace_{ij}$, i.\,e.\ $\chi^h(\reSpace_{ij})=\chi_{ij}$ for some $\chi_{ij}\in \Lip(I)$.
Furthermore, due to the single path property, the intersection of any two fibres $\chi_{ij}(I)$ and $\chi_{kl}(I)$
must be connected and can be assigned an orientation according to the irrigation direction.
Now define for any subset $S\subset\Z^n\times\Z^n$ the fibre intersection
\begin{displaymath}
f_S=\bigcap_{(i,j)\in S}\chi_{ij}(I)\setminus\bigcup_{(i,j)\notin S}\chi_{ij}(I)\,,
\end{displaymath}
where for simplicity we set $\chi_{ij}(I)=\emptyset$ for $\reSpace_{ij}=\emptyset$.
There are at most $2^N$ nonempty such intersections $f_S$, and each of them can have at most $N$ connected components $f_S^1,\ldots,f_S^N$
(again setting some of the $f_S^l$ to the empty set if necessary).
We have
\begin{displaymath}
\chi(\reSpace,I)=\bigcup_{S\subset\Z^n\times\Z^n}\bigcup_{0\leq l\leq N}f_S^l
\end{displaymath}
with at most $N2^N$ terms being nonempty.
Each of the $f_S^l$ can be assigned an orientation and a weight $w_S^l=\reMeasure\left(\bigcup_{(i,j)\subset S}\reSpace_{ij}\right)$,
the amount of particles travelling on $f_S^l$ (which is constant all along $f_S^l$).
Furthermore, $f_S^l$ must be a straight line segment, since otherwise, by straightening the fibres the cost of the irrigation pattern is reduced.
Hence, we can define a finite graph $G^h$ whose oriented edges are the $f_S^l$, whose vertices are the edge end points, and whose edge weights are the $w_S^l$.
It is now straightforward to check $\urbPlEn^{\varepsilon,a,\mu_+^h,\mu_-^h}[\chi^h]=\urbPlXia^{\varepsilon,a}(G^h)$ as required.
\end{proof}

The proof of the opposite inequality requires a few preparatory lemmas.

\begin{lemma}[Almost a $\Gamma$-convergence lemma]\label{lem:gamma_convergence_for_discrete_measures}
Suppose that
\begin{itemize}
 \item $\mu_+^N$, $\mu_-^N$ are discrete measures such that $\mu_+^N \weakstarto \mu_+$, $\mu_-^N \weakstarto \mu_-$ as $N\to\infty$;
 \item $\flux_N$ is a minimiser of $\urbPlEn^{\varepsilon,a,\mu_+^N,\mu_-^N}$;
 \item $\flux_N \weakstarto \flux$.
\end{itemize}
Then, $\flux$ is a minimiser of $\urbPlEn^{\varepsilon,a,\mu_+,\mu_-}$.
\end{lemma}

\begin{proof}
To achieve a contradiction suppose that $\flux$ is not a minimiser of $\urbPlEn^{\varepsilon,a,\mu_+,\mu_-}$, that is there exist $\flux'$ such that $\dv\flux' = \mu_+-\mu_-$ and $\urbPlEn^{\varepsilon,a,\mu_+,\mu_-}(\flux') < \urbPlEn^{\varepsilon,a,\mu_+,\mu_-}(\flux)$. On the right-hand side, we have (up to a subsequence)
\begin{displaymath}
 \urbPlEn^{\varepsilon,a,\mu_+,\mu_-}(\flux) \leq \liminf_{N \to \infty} \urbPlEn^{\varepsilon,a,\mu_+^N,\mu_-^N}(\flux_N)
\end{displaymath}
due to the weak-* lower semicontinuity of the energy.
Thus, given $\eta > 0$, for large $N$ we have
\begin{displaymath}
 \urbPlEn^{\varepsilon,a,\mu_+,\mu_-}(\flux) -\eta < \urbPlEn^{\varepsilon,a,\mu_+^N,\mu_-^N}(\flux_N).
\end{displaymath}
On the left-hand side, by definition of $\urbPlXia^{\varepsilon,a}$, there exists a sequence $\flux_N'$ such that
\begin{displaymath}
 \lim_{N \to \infty} \urbPlEn^{\varepsilon,a,\mu_+^N,\mu_-^N}(\flux_N') \leq \urbPlEn^{\varepsilon,a,\mu_+,\mu_-}(\flux') + \eta.
\end{displaymath}
Choosing $\eta$ suitably close to 0 now yields $\urbPlEn^{\varepsilon,a,\mu_+^N,\mu_-^N}(\flux_N')<\urbPlEn^{\varepsilon,a,\mu_+^N,\mu_-^N}(\flux_N)$ for $N$ large enough, a contradiction to the optimality of $\flux_N$.
\end{proof}

\begin{definition}[Paths in a graph and their weight]
Let $G$ be an oriented graph with edge set $E(G)$.
A \emph{path} in $G$ is a sequence $\xi = (e_1,\ldots,e_k)$ of edges $e_1,\ldots,e_k\in E(G)$ such that $e_i^- = e_{i-1}^+$ for $i = 2,\ldots,k$, where $e^-$ and $e^+$ denote the initial and final point of edge $e$.

Suppose that $G$ is also weighted and has no cycles.
Let us denote by $\Xi(G)$ the set of \emph{maximal paths} on $G$, that is, paths that are not a subsequence of any other path.

The \emph{weights} $w(\xi)$ of all paths $\xi\in\Xi(G)$ is defined by the system of equations
\begin{displaymath}
 w(e) = \sum_{e \in \xi} w(\xi), \quad e \in E(G),
\end{displaymath}
whose solvability follows from \cite[Lemma 7.1]{Xia-Optimal-Paths}.

Finally, for $\Xi_0 \subseteq \Xi(G)$ we define
\begin{displaymath}
 |\Xi_0| = \sum_{\xi \in \Xi_0} w(\xi).
\end{displaymath}
\end{definition}

\begin{lemma}[Bound on fibre length]\label{lem:bound_on_fibre_length}
Let $\mu_+,\mu_-\in\fbm(\R^n)$ be discrete measures of equal mass with support in a convex set $\Omega\subset\R^n$ and let $G$ be any discrete mass flux between $\mu_+$ and $\mu_-$.
Let $\Xi_0$ denote the set of paths in $G$ of length greater than $2a\diam\Omega$,
and let $G(\Xi_0)$ denote the graph whose associated vectorial measure is given by
\begin{displaymath}
 \flux_{G(\Xi_0)} = \sum_{\xi \in \Xi_0} \sum_{e \in \xi} w(\xi)\mu_e
\end{displaymath}
(cf.\ Definition\,\ref{def:graphs_as_vectorial_measures}).
Then there exists a discrete mass flux $G'$ between $\mu_+$ and $\mu_-$ such that all its paths have length bounded by $2a\diam\Omega$ and
\begin{align}
 \urbPlXia^{\varepsilon,a}(G)\!\!-\!\!\urbPlXia^{\varepsilon,a}(G') &\geq \|\flux_{G(\Xi_0)}\|_\rca - a\diam\Omega|\Xi_0|\geq a\diam\Omega|\Xi_0|\,,\label{eq:bound_on_fibre_length}\\
 \|\flux_G - \flux_{G'}\|_\rca &\leq \|\flux_{G(\Xi_0)}\|_\rca + a\diam\Omega|\Xi_0|\leq3(\urbPlXia^{\varepsilon,a}(G)\!\!-\!\!\urbPlXia^{\varepsilon,a}(G'))\,.\label{eq:fluxDiff}
\end{align}
\end{lemma}

\begin{proof}
Definitions\,\ref{def:graphs_as_vectorial_measures} and \ref{def:sums_of_graphs} can be extended to paths and thus allow to define graphs as sums of edges and of paths, which we make use of in the following to simplify the exposition.
Let
\begin{displaymath}
 G' = \sum_{\xi \in \Xi(G)\setminus\Xi_0} w(\xi)\xi + G'' = \sum_{e \in E(G)} w'(e)e + G'',
\end{displaymath}
where
\begin{displaymath}
 w'(e) = w(e) - \sum_{e \in \xi \in \Xi_0} w(\xi)
\end{displaymath}
and $G''$ is a graph composed of straight edges, which recovers the flux conservation condition $\dv G' = \dv G = \mu_+-\mu_-$.
Denoting the length of an edge $e$ by $l(e)$, we can now compute
\begin{align*}
 \urbPlXia^{\varepsilon,a}(G) - \urbPlXia^{\varepsilon,a}(G') &= \sum_{e \in E(G)} [c^{\varepsilon,a}(w(e))-c^{\varepsilon,a}(w'(e))]l(e) - \urbPlXia^{\varepsilon,a}(G'')\\
 &= \sum_{e \in E(G)} \left(\sum_{e \in \xi \in \Xi_0} w(\xi)\right)l(e) - \urbPlXia^{\varepsilon,a}(G'')\\
 &= \sum_{\xi \in \Xi_0} \sum_{e \in \xi} w(\xi)l(e) - \urbPlXia^{\varepsilon,a}(G'')\\
 &\geq \|\flux_{G(\Xi_0)}\|_\rca - a\diam\Omega|\Xi_0|\,.
\end{align*}
The relation $\|\flux_{G(\Xi_0)}\|_\rca=\sum_{\xi \in \Xi_0} w(\xi) \sum_{e \in \xi}l(e)\geq 2a\diam\Omega|\Xi_0|$ now concludes the proof of \eqref{eq:bound_on_fibre_length}.
Equation\,\eqref{eq:fluxDiff} directly follows from $G-G'=G(\Xi_0)-G''$.
\end{proof}

Finally we will need the following compactness lemma for transport path measures.

\begin{lemma}[Compactness for transport path measures]\label{lem:compactness_lemma_for_transport_path_measures}
Let $C>0$ and $\Omega\subset\R^n$ be compact, and consider the set
\begin{displaymath}
 \Theta_C =\left\{\theta\in\Theta\,:\,\theta(I)\subset\Omega\text{ and }\textstyle\int_I |\dot\theta(t)|\de t \leq C\right\}\subset \Theta\,.
\end{displaymath}
Let $\eta_N\in\TPM(\mu_+,\mu_-)$ be a sequence of transport path measures such that
\begin{equation*}
\eta_N(\Theta\setminus\Theta_C) = 0\,.
\end{equation*}
Then, up to a subsequence, $\eta_N \weakto \eta$ in the sense
\begin{displaymath}
 \int_{\Theta} \varphi(\theta)\,\de\eta_N(\theta) \to \int_{\Theta} \varphi(\theta)\,\de\eta(\theta) \quad\text{for all}\ \varphi\in\contbdd(\Theta)\,,
\end{displaymath}
where $\contbdd(\Theta)$ denotes the set of bounded continuous functions on $\Theta$.
Moreover, $\eta\in\TPM(\mu_+,\mu_-)$.
\end{lemma}

\begin{proof}
Note that $\Theta$ is separable (which follows from the separability of $\Lip(I)$) and that $\Theta_C$ is a (sequentially) compact subset of $\Theta$.
Indeed, let $\theta_n$, $n=1,2,\ldots$, be a sequence in $\Theta_C$.
Upon reparameterisation of each element (which does not change the sequence), the $\theta_n$ are uniformly Lipschitz.
Thus, by the Ascoli--Arzel\`a Theorem, up to a subsequence we have $\theta_n\to\theta\in\cont(I;\Omega)$.
Furthermore,
\begin{displaymath}
 \int_I |\dot\theta(t)|\,\de t \leq \liminf_N \int_I |\dot\theta_N(t)|\,\de t \leq C.
\end{displaymath}

As a consequence, the $\eta_N$ are all supported on the same compact set and are thus tight (i.e. for every $\varepsilon > 0$ there exists a compact $K_\varepsilon$ such that $\eta_N(K_\varepsilon^c) < \varepsilon$).
Furthermore, due to $\eta_N\in\TPM(\mu_+,\mu_-)$ they all have the same mass.
Hence, by Prokhorov's Theorem (which assures weak compactness for a tight set of measures; see \cite{Bil99}) we get $\eta_N \weakto \eta$ up to a subsequence, as desired.

It remains to prove $\pushforward{p_0}{\eta} = \mu_+$ (the proof of $\pushforward{p_1}{\eta} = \mu_-$ works analogously).
Because of $\pushforward{p_0}{\eta_N} = \mu_+$ for all $N$ we have
\begin{equation*}
 \int_\Theta \varphi(p_0(\theta))\,\de\eta_N(\theta)
 = \int_\Omega \varphi(x) \,\de\mu_+(x)
 \quad\text{for all}\ \varphi \in \contbdd(\Omega)\,.
\end{equation*}
Due to $\eta_N \weakto \eta$ as well as $\varphi\circ p_0\in\contbdd(\Theta)$, letting $N\to\infty$ we finally arrive at
\begin{equation*}
 \int_\Theta \varphi(p_0(\theta))\,\de\eta(\theta)
 = \int_\Omega \varphi(x) \,\de\mu_+(x) \quad\text{for all}\ \varphi \in \contbdd(\Omega)\,,
\end{equation*}
that is, $\pushforward{p_0}{\eta} = \mu_+$.
\end{proof}

\begin{proposition}\label{prop:constructPatternFromFlux}
We have
\begin{equation}\label{eq:urban_chi_leq_urban_flux}
 \min_\chi\urbPlEn^{\varepsilon,a,\mu_+,\mu_-}[\chi] \leq \min_\flux\urbPlEn^{\varepsilon,a,\mu_+,\mu_-}[\flux]\,.
\end{equation}
Furthermore, for any optimal mass flux  $\flux$ there is an optimal irrigation pattern $\chi$ so that both are related via
\begin{equation}\label{eq:constructPatternFromFlux}
\int_{\R^n}\varphi\cdot\de\flux=\int_\reSpace\int_I\varphi(\chi_p(t))\cdot\dot\chi_p(t)\,\de t\,\de \reMeasure(p) \;\text{for all}\; \varphi\in \cont_c(\R^n;\R^n).
\end{equation}
\end{proposition}
\begin{proof}
In the first part of the proof, we construct a pattern $\chi$ from an optimal flux.
So let the flux $\flux$ be optimal. We may assume $\urbPlEn^{\varepsilon,a,\mu_+,\mu_-}[\flux] < \infty$ since otherwise there is nothing to show by Remark \ref{rem:existenceUrbPlFiniteCostPattern}.
Let $G_N$ a sequence of finite weighted graphs such that
\begin{itemize}
  \item $G_N$ is a discrete mass flux between some $\mu_+^N$ and $\mu_-^N$,
  \item $(\mu_+^N,\mu_-^N,\flux_{G_N}) \weakstarto (\mu_+,\mu_-,\flux)$,
  \item $\urbPlXia^{\varepsilon,a}(G_N) \to \urbPlEn^{\varepsilon,a,\mu_+,\mu_-}[\flux]$.
\end{itemize}
Note that if $\hat G_N$ or rather $\flux_{\hat G_N}$ is a minimiser of $\urbPlEn^{\varepsilon,a,\mu_+^N,\mu_-^N}$, by Lemma \ref{lem:gamma_convergence_for_discrete_measures} we must have
\begin{displaymath}
 \lim_{N\to\infty} \urbPlXia^{\varepsilon,a}(G_N)-\urbPlXia^{\varepsilon,a}(\hat G_N) = 0.
\end{displaymath}

Let $\tilde G_N$ denote the cycle-reduced graph $G_N$. By Lemma\,\ref{lem:cycleCost}, $\urbPlXia^{\varepsilon,a}(\tilde G_N)\leq\urbPlXia^{\varepsilon,a}(G_N)$ and
\begin{equation*}
\|\flux_{G_N}-\flux_{\tilde G_N}\|_\rca
\leq \urbPlXia^{\varepsilon,a}(G_N)-\urbPlXia^{\varepsilon,a}(\tilde G_N)
\leq\urbPlXia^{\varepsilon,a}(G_N)-\urbPlXia^{\varepsilon,a}(\hat G_N)
\mathop\to_{N\to\infty}0\,.
\end{equation*}
Thus, without loss of generality, we may replace the $G_N$ by discrete mass fluxes without cycles, and from now on $G_N$ is supposed to be without cycles.

We may even assume the $G_N$ to only contain paths with length bounded by $2a\diam(\Omega)$, where $\Omega$ is a ball containing $\spt\mu_+$ and $\spt\mu_-$.
Indeed, let $G_N'$ and $\Xi_0^N$ be the graph and the set of paths from Lemma \ref{lem:bound_on_fibre_length} associated with $G_N$.
By Lemma \ref{lem:bound_on_fibre_length} we obtain
\begin{equation*}
\|\flux_{G_N}-\flux_{G_N'}\|_\rca
\leq 3(\urbPlXia^{\varepsilon,a}(G_N)-\urbPlXia^{\varepsilon,a}(G_N'))
\leq3(\urbPlXia^{\varepsilon,a}(G_N)-\urbPlXia^{\varepsilon,a}(\hat G_N))
\mathop\to_{N\to\infty}0\,.
\end{equation*}
Thus we may replace the $G_N$ by the $G_N'$, which have uniformly bounded path lengths.

Summarising, from now on we may assume the $G_N$ to have no cycles and to have path lengths bounded by $2a\diam\Omega$.
Hence, by Remark\,\ref{rem:fluxDecomp} there exist corresponding transport path measures $\eta_N\in\TPM(\mu_+,\mu_-)$.
Since they parametrise the graphs $G_N$, the $\eta_N$ have support on paths with lengths bounded by $2a\diam\Omega$ and images in $B_{2a\diam\Omega}(\Omega)$.
Thanks to Lemma \ref{lem:compactness_lemma_for_transport_path_measures}, we thus have (up to a subsequence) $\eta_N \weakto \eta$ for some $\eta\in\TPM(\mu_+,\mu_-)$.
By Skorokhod's Convergence Theorem \cite[App.\,A, Thm.\,A.8]{BeCaMo09}, there exist a sequence of irrigation patterns $\chi_N$ parameterising $\eta_N$ and an irrigation pattern $\chi$ parameterising $\eta$ such that
\begin{displaymath}
 \chi_N(p,\cdot) \stackrel{C^0(I)}{\longrightarrow} \chi(p,\cdot) \quad \text{for all }p \in \reSpace\,.
\end{displaymath}
By Proposition \ref{prop:urban_planning_energy_is_lower_semicontinuous} we have
\begin{equation*}
 \urbPlEn^{\varepsilon,a,\mu_+,\mu_-}[\chi] = \urbPlMMS^{\varepsilon,a}(\chi) 
 \leq \liminf_{N \to \infty} \urbPlMMS^{\varepsilon,a}(\chi_N) = \liminf_{N \to \infty} \urbPlXia^{\varepsilon,a}(G_N) = \urbPlEn^{\varepsilon,a,\mu_+,\mu_-}[\flux]\,,
\end{equation*}
and \eqref{eq:urban_chi_leq_urban_flux} is established.

In the second part of the proof we now explain the relation given by formula \eqref{eq:constructPatternFromFlux} between the constructed $\chi$ and $\flux$.
After reparameterisation according to Proposition\,\ref{thm:constSpeedPatternsUrbPl} we may assume the $\chi_N(p,\cdot)$ to be uniformly Lipschitz with constant $2a\diam(\Omega)$.
Thus, for each $p$ we may extract a subsequence converging weakly-$*$ in $W^{1,\infty}(I)$ against $\chi(p,\cdot)$.
Since any subsequence contains such a converging subsequence, actually the whole sequence $\chi_N(p,\cdot)$ converges weakly-$*$ against $\chi(p,\cdot)$.
Now for any $\varphi\in \cont_c(\R^n;\R^n)$ we have (the second equality can be easily verified edge by edge)
\begin{displaymath}
\int_\Omega\varphi\cdot\de\flux
=\lim_{N\to\infty}\int_\Omega\varphi\cdot\de\flux_{G_N}
=\lim_{N\to\infty}\int_\reSpace\int_I\varphi(\chi_N(p,t))\cdot\dot\chi_N(p,t)\,\de t\,\de \reMeasure(p)\,.
\end{displaymath}
Note that $\varphi(\chi_N(p,\cdot))$ converges in $L^\infty(I)$, while $\dot\chi_N(p,\cdot)$ converges weakly-$*$ in $L^\infty(I)$ so that
\begin{displaymath}
\int_I\varphi(\chi_N(p,t))\cdot\dot\chi_N(p,t)\,\de t=:J_N(p)\to J(p):=\int_I\varphi(\chi(p,t))\cdot\dot\chi(p,t)\,\de t\,.
\end{displaymath}
Together with the uniform bound $J_N(p)\leq\|\varphi\|_{L^\infty}\hdone(\chi_N(p,I))\leq2\|\varphi\|_{L^\infty}a\diam(\Omega)$,
this allows application of Lebesgue's dominated convergence theorem, from which we finally obtain
\begin{displaymath}
\int_\Omega\varphi\cdot\de\flux
=\lim_{N\to\infty}\int_\reSpace\int_I\varphi(\chi_N(p,t))\cdot\dot\chi_N(p,t)\,\de t\,\de \reMeasure(p)
=\int_\reSpace\int_I\varphi(\chi(p,t))\cdot\dot\chi(p,t)\,\de t\,\de \reMeasure(p)\,,
\end{displaymath}
the desired formula.
\end{proof}

\section{Acknowledgements}
This work was supported by the Deutsche Forschungsgemeinschaft (DFG), Cells-in-Motion Cluster of Excellence (EXC 1003-CiM), University of M\"unster, Germany.
B.W.'s research was supported by the Alfried Krupp Prize for Young University Teachers awarded by the Alfried Krupp von Bohlen und Halbach-Stiftung.

\bibliographystyle{alpha}
\bibliography{BrWi14}

\begin{thebibliography}{BPSS09}

\bibitem[AGS08]{Ambrosio-Gigli-Sarave-Gradient-Flows}
Luigi Ambrosio, Nicola Gigli, and Giuseppe Savar{\'e}.
\newblock {\em Gradient flows in metric spaces and in the space of probability
  measures}.
\newblock Lectures in Mathematics ETH Z\"urich. Birkh\"auser Verlag, Basel,
  second edition, 2008.

\bibitem[AMO93]{Ahuja-Magnanti-Orlin}
Ravindra~K. Ahuja, Thomas~L. Magnanti, and James~B. Orlin.
\newblock {\em Network flows}.
\newblock Prentice Hall, Inc., Englewood Cliffs, NJ, 1993.
\newblock Theory, algorithms, and applications.

\bibitem[BB05]{Brancolini-Buttazzo}
Alessio Brancolini and Giuseppe Buttazzo.
\newblock Optimal networks for mass transportation problems.
\newblock {\em ESAIM Control Optim. Calc. Var.}, 11(1):88--101 (electronic),
  2005.

\bibitem[BB10]{Bianchini-Brancolini}
Stefano Bianchini and Alessio Brancolini.
\newblock Estimates on path functionals over {W}asserstein spaces.
\newblock {\em SIAM J. Math. Anal.}, 42(3):1179--1217, 2010.

\bibitem[BBS06]{Brancolini-Buttazzo-Santambrogio}
Alessio Brancolini, Giuseppe Buttazzo, and Filippo Santambrogio.
\newblock Path functionals over {W}asserstein spaces.
\newblock {\em J. Eur. Math. Soc. (JEMS)}, 8(3):415--434, 2006.

\bibitem[BCM05]{Bernot-Caselles-Morel-Traffic-Plans}
Marc Bernot, Vicent Caselles, and Jean-Michel Morel.
\newblock Traffic plans.
\newblock {\em Publ. Mat.}, 49(2):417--451, 2005.

\bibitem[BCM08]{Bernot-Caselles-Morel-Structure-Branched}
Marc Bernot, Vicent Caselles, and Jean-Michel Morel.
\newblock The structure of branched transportation networks.
\newblock {\em Calc. Var. Partial Differential Equations}, 32(3):279--317,
  2008.

\bibitem[BCM09]{BeCaMo09}
Marc Bernot, Vicent Caselles, and Jean-Michel Morel.
\newblock {\em Optimal transportation networks}, volume 1955 of {\em Lecture
  Notes in Mathematics}.
\newblock Springer-Verlag, Berlin, 2009.
\newblock Models and theory.

\bibitem[Bil99]{Bil99}
Patrick Billingsley.
\newblock {\em Convergence of probability measures}.
\newblock Wiley Series in Probability and Statistics: Probability and
  Statistics. John Wiley \& Sons, Inc., New York, second edition, 1999.
\newblock A Wiley-Interscience Publication.

\bibitem[BPSS09]{BuPrSoSt09}
Giuseppe Buttazzo, Aldo Pratelli, Sergio Solimini, and Eugene Stepanov.
\newblock {\em Optimal urban networks via mass transportation}, volume 1961 of
  {\em Lecture Notes in Mathematics}.
\newblock Springer-Verlag, Berlin, 2009.

\bibitem[BS11a]{Brancolini-Solimini-Hoelder}
Alessio Brancolini and Sergio Solimini.
\newblock On the {H}\"older regularity of the landscape function.
\newblock {\em Interfaces Free Bound.}, 13(2):191--222, 2011.

\bibitem[BS11b]{Brasco-Santambrogio}
Lorenzo Brasco and Filippo Santambrogio.
\newblock An equivalent path functional formulation of branched transportation
  problems.
\newblock {\em Discrete Contin. Dyn. Syst.}, 29(3):845--871, 2011.

\bibitem[BS14]{Brancolini-Solimini-Fractal}
Alessio Brancolini and Sergio Solimini.
\newblock Fractal regularity results on optimal irrigation patterns.
\newblock {\em J. Math. Pures Appl. (9)}, 102(5):854--890, 2014.

\bibitem[Dud02]{Dudley}
Richard~M. Dudley.
\newblock {\em Real analysis and probability}, volume~74 of {\em Cambridge
  Studies in Advanced Mathematics}.
\newblock Cambridge University Press, Cambridge, 2002.
\newblock Revised reprint of the 1989 original.

\bibitem[GDL14]{Gauthier-Desrosiers-Luebbecke}
Jean~Bertrand Gauthier, Jacques Desrosiers, and Marco~E. L{\"u}bbecke.
\newblock Decomposition theorems for linear programs.
\newblock {\em Oper. Res. Lett.}, 42(8):553--557, 2014.

\bibitem[MS09]{Maddalena-Solimini-Transport-Distances}
Francesco Maddalena and Sergio Solimini.
\newblock Transport distances and irrigation models.
\newblock {\em J. Convex Anal.}, 16(1):121--152, 2009.

\bibitem[MS10]{Morel-Santambrogio-Regularity}
Jean-Michel Morel and Filippo Santambrogio.
\newblock The regularity of optimal irrigation patterns.
\newblock {\em Arch. Ration. Mech. Anal.}, 195(2):499--531, 2010.

\bibitem[MS13]{Maddalena-Solimini-Synchronic}
Francesco Maddalena and Sergio Solimini.
\newblock Synchronic and asynchronic descriptions of irrigation problems.
\newblock {\em Adv. Nonlinear Stud.}, 13(3):583--623, 2013.

\bibitem[MSM03]{Maddalena-Morel-Solimini-Irrigation-Patterns}
F.~Maddalena, S.~Solimini, and J.-M. Morel.
\newblock A variational model of irrigation patterns.
\newblock {\em Interfaces Free Bound.}, 5(4):391--415, 2003.

\bibitem[Roy88]{Royden-Real-Analysis}
Halsey~L. Royden.
\newblock {\em Real analysis}.
\newblock Macmillan Publishing Company, New York, third edition, 1988.

\bibitem[Rud87]{Ru87}
Walter Rudin.
\newblock {\em Real and complex analysis}.
\newblock McGraw-Hill Book Co., New York, third edition, 1987.

\bibitem[San07]{Santambrogio-Landscape}
Filippo Santambrogio.
\newblock Optimal channel networks, landscape function and branched transport.
\newblock {\em Interfaces Free Bound.}, 9(1):149--169, 2007.

\bibitem[Vil03]{Villani-Topics-Optimal-Transport}
C{\'e}dric Villani.
\newblock {\em Topics in optimal transportation}, volume~58 of {\em Graduate
  Studies in Mathematics}.
\newblock American Mathematical Society, Providence, RI, 2003.

\bibitem[Vil09]{Villani-Transport-Old-New}
C{\'e}dric Villani.
\newblock {\em Optimal transport}, volume 338 of {\em Grundlehren der
  Mathematischen Wissenschaften [Fundamental Principles of Mathematical
  Sciences]}.
\newblock Springer-Verlag, Berlin, 2009.
\newblock Old and new.

\bibitem[Xia03]{Xia-Optimal-Paths}
Qinglan Xia.
\newblock Optimal paths related to transport problems.
\newblock {\em Commun. Contemp. Math.}, 5(2):251--279, 2003.

\bibitem[Xia04]{Xia-Interior-Regularity}
Qinglan Xia.
\newblock Interior regularity of optimal transport paths.
\newblock {\em Calc. Var. Partial Differential Equations}, 20(3):283--299,
  2004.

\end{thebibliography}

\end{document}